\documentclass[reqno,11pt]{amsart}
 \newtheorem{thm}{Theorem}[section]
 
 \newtheorem{lem}[thm]{Lemma}
 \newtheorem{prop}[thm]{Proposition}
 \theoremstyle{definition}
 \newtheorem{defn}[thm]{Definition}
 \theoremstyle{remark}
 \newtheorem{rem}[thm]{Remark}
 \newtheorem{ex}{Example}
 \numberwithin{equation}{section}
\usepackage{color}
\begin{document}
%
%
%
%
%
%
%
%
%
\title[On Some Geometric Properties of Slice Regular Functions]
{On Some Geometric Properties of Slice Regular Functions of a Quaternion Variable}

\author[Sorin G. Gal]{Sorin G. Gal}
\address{
University of Oradea\\
Department of Mathematics and Computer Science\\
Str. Universitatii Nr. 1\\
410087 Oradea\\
Romania}
\email{galso@uoradea.ro}
\author[J. Oscar Gonz\'alez-Cervantes]{J. Oscar Gonz\'alez-Cervantes}
\address{
 Departamento de Matem\'aticas  \\
 E.S.F.M. del
I.P.N. 07338 \\
M\'exico D.F., M\'exico}
\email{
jogc200678@gmail.com}

\author[Irene Sabadini]{Irene Sabadini}
\address{%
Dipartimento di Matematica\\Politecnico di Milano\\Via Bonardi,
9\\20133 Milano, Italy}
\email{irene.sabadini@polimi.it}




\begin{abstract}
The goal of this paper is to introduce and study some geometric properties of slice regular functions of quaternion variable like
 univalence, subordination, starlikeness, convexity and spirallikeness in the unit ball. We prove a number of results, among which an Area-type Theorem, Rogosinski inequality, and a Bieberbach-de Branges Theorem for a subclass of slice regular functions. We also discuss some geometric and algebraic interpretations of our results in terms of maps from $\mathbb R^4$ to itself. As a tool for subordination we define a suitable notion of composition of slice regular functions which is of independent interest.
\end{abstract}

\maketitle

\section{Introduction}

The functions we consider in this paper are  power series of the quaternion variable $q$ of the form $\sum_{n=0}^\infty q^n a_n$ with quaternionic coefficients $a_n$, converging in the unit ball $\mathbb B$. These functions are (left) slice regular according to the definition in \cite{GS} and also according to the definition in \cite{GP}. The two definitions in \cite{GS} and in \cite{GP} are different, but they give rise to the same class of functions on some particular opens sets called axially symmetric slice domains, that we will introduce in the next section. Slice regular functions are nowadays a widely studied topic, important especially for its application to a functional calculus for quaternionic linear operators, see \cite{CSS}, and to Schur analysis, see \cite{acs1}, and \cite{acs2} in which Blaschke factors are also studied.

It is then natural to continue the study of this class of functions by
 considering some geometric properties of slice regular functions like univalence, subordination, starlikeness, convexity and spirallikeness in the unit ball.
In the literature, some other geometric properties of this class of functions have been already considered and some results have been proved: the Bloch-Landau theorem, \cite{DRGS}, the Bohr theorem, \cite{GS1}, Landau-Toeplitz theorem, \cite{DRGS1}, (some of these results are collected in \cite{GSS}) Schwarz-Pick lemma, see \cite{BS} and \cite{abcs} for an alternative, shorter proof. Recently also the Borel-Carath\'eodory theorem has been proved, see \cite{gr} and also \cite{acks} for a weaker version.

The plan of the paper is as follows. Section 2 contains some known concepts and results useful in the next sections. Section 3 discusses univalence of slice regular functions and several conditions under which a function defined on the open unit ball of the quaternions is univalent. We introduce the counterpart of the Koebe function in the quaternionic setting and obtain an Area-type theorem and a Bieberbach- de Branges result for a subclass of univalent slice regular functions. We also provide some algebraic and geometric interpretation of our result in terms of transformations from $\mathbb R^4$ to itself. In Section 4 we introduce a notion of composition of formal power series and we study their convergence. We then introduce the notion of subordination of slice regular functions and we prove the Rogosinski inequality. In Section 5 we consider starlike and convex slice regular functions, also discussing some geometric consequences. Finally, in Section 6, we consider spirallike functions.

\section{Preliminaries}

Let us recall that the quaternion field is defined as
$$\mathbb{H}=\{q=x_1  + x_{2} i +x_{3} j + x_{4} k ; x_1, x_2, x_3, x_4 \in \mathbb{R} \},$$
where the imaginary units $i, j, k\not \in \mathbb{R}$ satisfy
$$i^2=j^2=k^2=-1, \ ij=-ji=k, \ jk=-kj=i, \ ki=-ik=j.$$ It is a noncommutative field and since obviously $\mathbb{C}\subset \mathbb{H}$, it extends the class of complex numbers. On $\mathbb{H}$ can be defined the norm
$\|q\|=\sqrt{x_1^2 +x_2^2+x_3^3+x_4^2}$, for $q=x_1  + x_{2} i +x_{3} j + x_{4} k$.

Let us denote by $\mathbb{S}$ the unit sphere of purely imaginary quaternion, i.e.
$$\mathbb{S}=\{q = i x_{1} + j x_{2} + k x_{3}, \mbox{ such that } x_{1}^{2}+x_{2}^{2}+x_{3}^{3}=1\}.$$
Note that if $I\in \mathbb{S}$, then $I^{2}=-1$. For this reason the elements of $\mathbb{S}$ are also
called imaginary units. For any fixed $I\in\mathbb{S}$ we define $\mathbb{C}_I:=\{x+Iy; \ |\ x,y\in\mathbb{R}\}$. It is easy to verify that
$\mathbb{C}_I$ can be identified with a complex plane, moreover $\mathbb{H}=\bigcup_{I\in\mathbb{S}} \mathbb{C}_I$.
The real axis belongs to $\mathbb{C}_I$ for every $I\in\mathbb{S}$ and thus a real quaternion can be associated to any imaginary unit $I$.
Any non real quaternion $q$ is uniquely associated to the element $I_q\in\mathbb{S}$
defined by $I_q:=( i x_{1} + j x_{2} + k x_{3})/\|  i x_{1} + j x_{2} + k x_{3}\|$ and, obviously, $q$ belongs to the complex plane $\mathbb{C}_{I_q}$.

Also, recall that for $q\in \mathbb{H}\setminus \mathbb{R}$, $q=x_{1}+i x_{2}+j x_{3}+ k x_{4}$, defining
$r:=\|q\|=\sqrt{x_{1}^{2}+x_{2}^{2}+x_{3}^{2}+x_{4}^{2}}$, there exists uniquely $a\in (0, \pi)$ with $cos(a):=\frac{x_{1}}{r}$ and there
exists uniquely $I_{q}\in \mathbb{S}$, such that
$$q=r e^{I_{q} a}, \, \mbox{ with } I_{q}=i y + j v + k s,\, y=\frac{x_{2}}{r \sin(a)},\, u=\frac{x_{3}}{r \sin(a)}, \, s=\frac{x_{4}}{r \sin(a)}.$$
Now, if $q\in \mathbb{R}$, then we choose $a=0$, if $q>0$ and $a=\pi$ if $q<0$, and as $I_{q}$ we choose an arbitrary fixed $I\in \mathbb{S}$.
So that if $q\in \mathbb{R}\setminus \{0\}$, then again we can write $q=\|q\|(\cos(a)+I\sin(a))$ (but with non-unique $I$).

The above is called the trigonometric form of the quaternion number $q\not=0$ and $a$ denoted by ${\rm arg}(q)$ is called the argument of the quaternion $q$.

{\it Evidently, $"a"$ could be considered as the angle between the real axis and the segment $[0, q]$ in $\mathbb{R}^{4}$ (or, in other words,
the angle between the real axis and the radius in $\mathbb{R}^{4}$ passing through origin and the geometric image in $\mathbb{R}^{4}$ of $q$).}

If $q=0$, then we do not have a trigonometric form for $q$ (exactly as in the complex case).

For our purposes we will need the following concept of analyticity of functions of a quaternion variable.

\begin{defn}
 Let $U$ be an open set in $\mathbb{H}$ and $f:U\to \mathbb{H}$ real differentiable. $f$ is called left slice regular or slice hyperholomorphic
if for every $I\in \mathbb{S}$, its restriction $f_{I}$ to the complex plane ${\mathbb{C}}_{I}=\mathbb{R}+ I \mathbb{R}$ passing through origin
and containing $I$ and $1$ satisfies
$$\overline{\partial}_{I}f(x+I y):=\frac{1}{2}\left (\frac{\partial}{\partial x}+I \frac{\partial}{\partial y}\right )f_{I}(x+I y)=0,$$
on $U\bigcap \mathbb{C}_{I}$.
The class of slice regular functions on $U$ will be denoted by $\mathcal{R}(U)$.
\end{defn}
Let $f\in\mathcal{R}(U)$. The so called left (slice) $I$-derivative of $f$ at a point $q=x+I y$ is given by
$$\partial_{I}f_{I}(x+ I y):=\frac{1}{2}\left (\frac{\partial}{\partial x}f_{I}(x+ I y) - I\frac{\partial}{\partial y}f_{I}(x+ I y)\right ).$$

Analogously, a function is called right slice regular if
$$(f_{I}{\overline{\partial}}_{I})(x+I y):=\frac{1}{2}\left (\frac{\partial}{\partial x}f_{I}(x +I y)+\frac{\partial}{\partial y}f_{I}(x+I y) I\right )=0,$$
on $U\bigcap \mathbb{C}_{I}$.

In this case, the right $I$-derivative of $f$ at $q=x+I y$ is given by
$$\partial_{I}f_{I}(x+ I y):=\frac{1}{2}\left (\frac{\partial}{\partial x}f_{I}(x+ I y) - \frac{\partial}{\partial y}f_{I}(x+ I y) I\right ).$$

 Let us now introduce a suitable notion of  derivative:
\begin{defn}
Let $U$ be an open set in $\mathbb{H}$, and let $f:U \to \mathbb{H}$
be a slice regular function. The slice derivative of $f$,
$\partial_s f$, is defined by:
\begin{displaymath}
\partial_s(f)(q) = \left\{ \begin{array}{ll}
\partial_I(f)(q) & \textrm{ if $q=x+Iy$, \ $y\neq 0$},\\ \\
\displaystyle\frac{\partial f}{\partial x} (x) & \textrm{ if\  $q=x\in\mathbb{R}$.}
\end{array} \right.
\end{displaymath}
\end{defn}
The definition of slice derivative is well posed because it is applied
only to slice regular functions and thus
$$
\frac{\partial}{\partial x}f(x+Iy)=
-I\frac{\partial}{\partial y}f(x+Iy)\qquad \forall I\in\mathbb{S},
$$
and therefore, analogously to what happens in the complex case,
$$ \partial_s(f)(x+Iy) =
\partial_I(f)(x+Iy)=\partial_x(f)(x+Iy).
$$
We will often  write $f'(q)$ instead of $\partial_s f(q)$.\\
If $f$ is a slice regular function, then also its slice derivative is slice regular, in fact
\begin{equation}
\label{derivatareg}\overline{\partial}_I(\partial_sf(x+Iy))
=\partial_s(\overline{\partial}_If(x+Iy))=0,\nonumber
\end{equation}
and therefore
$$
\partial^n_sf(x+Iy)=\frac{\partial^n
f}{\partial x^n}(x+Iy).
$$

Among the useful tools from the general theory on slice regular functions, we  recall some useful facts, collected in a following theorem. For the definition of axially symmetric s-domain we refer  the reader e.g. to \cite{CSS}. For our purposes, it is enough to know that balls in $\mathbb{H}$ are examples of axially symmetric s-domains.

\begin{thm}\label{1.3}
Let $U \subseteq\mathbb{H}$ be an axially symmetric s-domain and $f:U\to \mathbb{H}$ be a (left) slice regular function.
\begin{enumerate}
\item (Representation Formula) The following equality holds for all $q=x+I y \in \Omega$:
\begin{equation}\label{strutturaquat1}
f(x+I y) = \frac 12[   f(x+Jy)+ f(x-Jy)]
+ \frac 12 IJ[ f(x-Jy)-f(x+Jy)].
\end{equation}
\item (General Representation Formula) The following equality holds for all $q=x+I y \in \Omega$:
\begin{equation}\label{strutturaquat2}
f(x+I y) =(J-K)^{-1}[ J  f(x+Jy)- Kf(x+Ky)]
+I(J-K)^{-1}[ f(x+Jy)-f(x+Ky)].
\end{equation}
 \item (Splitting Lemma, see. e.g. Lemma 4.1.7, p. 117 in \cite{CSS}) For every $I\in \mathbb{S}$ and every $J\in \mathbb{S}$, perpendicular to $I$,
there are two holomorphic functions $F, G : U\bigcap \mathbb{C}_{I}\to \mathbb{C}_{I}$, such that for any $z=x+I y$ we have
$f_{I}(z)=F(z)+G(z) J$.

 \item (see e.g. Corollary 4.3.6, p. 121 in \cite{CSS}) For all $x, y\in \mathbb{R}$ such that $x+I y\in U$, there exist $a, b\in \mathbb{H}$
such that $f(x+ I y)=\alpha(x,y) + I \beta(x,y)$, for all $I\in \mathbb{S}$.

 \item (see e.g. Corollary 4.3.4, p. 121, \cite{CSS}) Define $D\subset \mathbb{R}^{2}$ such that $(x, y)\in D$ implies $x+I y\in U$.
The function $f$ is slice regular if and only if there exist two differentiable functions
$\alpha, \beta : D\subset \mathbb{R}^{2}\to \mathbb{H}$ satisfying on $D$, $\alpha(x, y)=\alpha(x, -y)$, $\beta(x, y)=-\beta(x, -y)$ and the Cauchy-Riemann system $\frac{\partial \alpha}{\partial x}-\frac{\partial \beta}{\partial y}=0$,
$\frac{\partial \beta}{\partial x}+\frac{\partial \alpha}{\partial y}=0$, such that
$f(x+ Iy)=\alpha(x, y)+ I\beta(x, y)$. The functions $\alpha(x, y)$ and $\beta(x, y)$ do not depend on $I\in \mathbb{S}$.

 \item (See e.g. Theorem 2.7 in \cite{GS}) Let ${B}(0;{R})=\{q\in \mathbb{H} ; \|q\| < R\}$. A function $f:\ {B}(0;{R}) \to \mathbb{H}$ is (left) slice regular on ${B}(0;{R})$ if and only if it has a series representation of the form
$$f(q)=\sum_{n=0}^{\infty}q^{n}\frac{1}{n !}\cdot \frac{\partial^{n} f}{\partial x^{n}}(0),$$
uniformly convergent on ${B}(0;{R})$.
\end{enumerate}
\end{thm}

A useful subclass of slice regular functions is denoted by the letter $\mathcal{N}$ (see e.g. \cite{CSS}, p. 152, Definition 4.11.2) and it can be characterized in various ways.
Let $U$ be an open set in $\mathbb{H}$ and define
$$
 \mathcal{N}(U)=\{ f \ {\rm slice\ regular\ in}\ U\ :  \ f(U\cap \mathbb{C}_I)\subseteq  \mathbb{C}_I,\ \  \forall I\in \mathbb{S}\}.
$$
It is possible to prove  that if $U$ is an axially symmetric s-domain, then $f\in\mathcal{N}(U)$ if and only if it is of the form described in Theorem \ref{1.3} (3)
where $\alpha$, $\beta$ are real valued.\\
 Let us consider the ball $B(0;R)$ with center at the origin and radius $R>0$; it is immediate that  a function slice regular on $B(0;R)$ belong to $\mathcal{N}$ if and only if its
power series expansion has real coefficients. Thus the class $\mathcal{N}$ includes all elementary transcendental functions, like exponential, logarithm, sine, cosine, etc.\\
Finally, if we denote by $Z_{\mathbb{H}}$ the conjugate quaternion, that is $Z_{\mathbb{H}}(q)=\overline{q}$, it can be shown, see \cite{CGS}, that a function $f\in\mathcal{N}(U)$, where $U$ is an axially symmetric s-domain,
if and only is it satisfies $f=Z_{\mathbb{H}}\circ f\circ Z_{\mathbb{H}}$. This property is called in \cite{CGS} C-property, where "C" stands for conjugation.
In analogy to the terminology in the complex case, where functions satisfying $f(\bar z)=\overline{f(z)}$ or, equivalently, $f(z)=\overline{f(\bar z)}$  are called {\em intrinsic}, we will call these functions
{\em quaternionic intrinsic}.
Quaternionic intrinsic functions are the building blocks of slice regular functions in the sense of the following result:
\begin{prop} (See Proposition 3.12 in \cite{CGS})
Let $U$ be an axially symmetric s-domain and let $\{1,{\bf i},{\bf j},{\bf i}{\bf j}\}$ be a basis of $\mathbb{H}$, as a real vector space. Then the right vector space
 $\mathcal{R}(U)$ of slice regular functions on $U$ can be decomposed as:
$$\mathcal{R}(U)=\mathcal{N}(U)  \oplus \mathcal{N}(U){\bf i}   \oplus  \mathcal{N}(U)  {\bf j}\oplus \mathcal{N}(U) \bf{ij}.
$$
\end{prop}
Functions belonging  to the class $\mathcal{N}$ have nice properties, see for example \cite{CS} for applications to the spectral mapping theorem. It is also important to mention that, in general, the composition of two slice regular functions is not slice regular. However we have
\begin{thm} Let $f\in\mathcal R(U)$ and $g\in\mathcal N(V)$ be such that $g(V)\subseteq U$. Then $f\circ g\in\mathcal R(V)$.
\end{thm}

\section{Univalence of Slice Regular Functions}

\begin{defn}Let $f\in\mathcal {R}(B(0;1) )$, where $B(0;1)=\{q\in\mathbb H\, ;\  \| q\|<1\}$, then :
\begin{enumerate}
\item $f$ is called  univalent in $B(0;1)$ if it is injective in $B(0;1)$ ;
\item For $I\in \mathbb S$,  $f$ is $I$-univalent if $f\mid_{\mathbb D_I}$ is univalent, where  $\mathbb D_I =B(0;1)\cap \mathbb C_I$.
\end{enumerate}
\end{defn}
\begin{rem}\label{rem001} {} \ \begin{enumerate}
\item  From the definition it is clear that if $f\in\mathcal{R}(B(0;1)) $ is univalent then f is $I$-univalent  for any $I\in\mathcal S$.
    \\
    Consider  the slice regular function  $$f(q)=q^2 + qJ, \quad q\in  B(0;1),$$
    where  $J\in \mathbb S$ is a fixed element.
Then $f$ is $I$-univalent for any $I\in \mathbb S\setminus\{J\}$ and for $z\in \mathbb D_J$ the restriction
$$f\mid_{\mathbb D_J}(z)=z^2 + zJ, \quad z\in \mathbb D_J,$$
is not univalent. We conclude that $f$ is not $J$-univalent and therefore is not univalent in $B(0;1)$.

\item Consider $f\in\mathcal{N}(B(0;1)) $ and  suppose  $I\in \mathbb S$ such that $f$ is an $I$-univalent function.  As $f$ can be expressed by $f(q) = \alpha(x,y) + I \beta (x,y )$, for all $q=x+I y \in \mathbb B(0;1) $,   where $\alpha $ and $\beta$ are real valued functions,  then $\alpha$ and $\beta$ satisfy the condition : if $x_{i}^{2}+y_{i}^{2}<1$, $i=1, 2$ and $(x_{1}, y_{1})\neq (x_{2}, y_{2})$  then $\alpha(x_{1}, y_{1})\neq \alpha(x_{2}, y_{2})$ or $\beta(x_{1}, y_{1})\neq \beta(x_{2}, y_{2})$.
    This clearly implies that $f$ is $J$-univalent for all $J\in \mathbb S$. Finally, if $w,q$ belong to different slices  then  $f(w)\neq f(q)$. Therefore $f$ is univalent too. For example, the functions   $$f(q) = q  (1-q)^{-1}, \quad g (q) = q  (1-q^2)^{-1}, \quad \quad h(q) = q - \frac{1}{2}q^2,  \quad q\in \mathbb B(0;1)$$
    are univalent.
\end{enumerate} \end{rem}

Suppose that $f:B(0 ; 1)\to \mathbb{H}$ is a slice (left) regular function, i.e. $f(q)=\sum_{k=0}^{\infty}q^{k}a_{k}$,
for all $q\in B(0 ; 1)$.
A natural question is if the condition ${\rm Re}[\partial_{s}(f)(q)]>0$, for all $q\in B(0;1)$ implies that $f$ is univalent (injective) on $B(0;1)$,
as it happens in the complex variable case.

Partial answers to this question are the following results.

\begin{thm}\label{2.2}
Let $f:B(0 ; 1)\to \mathbb{H}$ be a slice (left) regular function, i.e. $\displaystyle f(q)=\sum_{k=0}^{\infty}q^{k}a_{k}$, for all
$q\in B(0 ; 1)$. Suppose also that $Re[\partial_{s}(f)(q)]>0$ for all $q\in B(0 ; 1)$. Then one has the following properties
\begin{enumerate}
\item For any $I\in \mathbb{S}$, $f$ is univalent on $B(0;1)\bigcap \mathbb{C}_{I}$.
\item   Suppose that  $a_0=0  $  and  $a_1=1$, then $\displaystyle \frac{1- \|q\|}{ 1+\|q\|} \leq {\rm Re}[\partial_s f(q)] \leq \|\partial_s f(q)\|\leq \frac{1+\|q\|}{1-\|q\|}, \quad q\in B(0;1).$
\\
If the first or the third inequality  becomes in  an equality for  some $q\neq 0$  then  there exist  $I\in \mathbb S$ and  $\theta\in\mathbb R$ such that
$$  f(q) =q + \sum_{n=2}^{\infty} q^n \frac{2 e^{I(n-1) \theta } }{n}, \quad q\in B(0;1).$$
Moreover,   $\|na_n\|\leq 2$ for each $n\in\mathbb N$. But if   for some $n_0\in\mathbb N$ one has that $\|a_{n_0}\|=\frac{2}{n_0} $  if and only if $$ \displaystyle a_{kn_0 +1} = \frac{ 2}{kn_0 +1} (   \frac{n_0a_{n_0}}{ 2} )^k$$ for each $k\in \mathbb N$.   Note that if $\|a_2\|=1$   then there exist $I\in \mathbb S$ and $\theta\in\mathbb R$ such that
$$ f(q) = q + \sum_{n=2}^{\infty}  q^n \frac{2 e^{I(n-1) \theta } }{n}, \quad q\in B(0;1).$$
\end{enumerate}
\end{thm}
\begin{proof}
\begin{enumerate}
\item  Let $z_{1}, z_{2}\in \mathbb{H}$ be with $z_{1}\not=z_{2}$. Denoting
$z(t)=t z_{2}+(1-t)z_{1}$, $t\in [0, 1]$, we have
$$\frac{d[(z(t))^{k}]}{d t}=\sum_{p=0}^{k-1}z(t)^{p}\cdot z_{0}\cdot z(t)^{k-1-p}, \, \mbox{ where } z_{0}=z_{2}-z_{1}, k\ge 1.$$
Let $z_{1}, z_{2}\in B(0;1)\bigcap \mathbb{C}_{I}$ be with an arbitrary $I\in \mathbb{S}$, $z_{1}\not=z_{2}$.
Since obviously
$z_{0}\in \mathbb{C}_{I}$ and $z(t)\in \mathbb{C}_{I}$ for all $t\in [0, 1]$, due to the commutativity (in $\mathbb{C}_{I}$) we easily get $\frac{d[(z(t))^{k}]}{d t}=z_{0}\cdot k z(t)^{k-1}$
which (due to the uniform convergence of the series) immediately implies that
$$\frac{d [f(z(t))]}{d t}=\sum_{k=0}^{\infty}\frac{d[(z(t))^{k}]}{d t} a_{k}=z_{0}\cdot \sum_{k=1}^{\infty}k z(t)^{k-1}a_{k}
=z_{0} \cdot \partial_{s}(f)(z(t)).$$
Reasoning now as in the complex variable case, we get
$$f(z_{2})-f(z_{1})=\int_{0}^{1}\frac{d [f(z(t))]}{d t} d t = z_{0}\int_{0}^{1}\partial_{s}(f)(z(t)) d t.$$
Since $z_{0}\not=0$, this implies
$z_{0}^{-1}\cdot [f(z_{2})-f(z_{1})]=\int_{0}^{1}\partial_{s}(f)(z(t)) d t$ and
$${\rm Re} [z_{0}^{-1}\cdot [f(z_{2})-f(z_{1})]={\rm Re}\left [\int_{0}^{1}\partial_{s}(f)(z(t)) d t\right ]>0,$$
that is $f(z_{2})-f(z_{1})\not=0$.
\item From the Carath\'eodory theorem for slice regular functions, see \cite[Theorem 3]{gr},  $0 < {\rm Re}\, \partial_s f (q)$  for each $q\in B(0;1) $  implies
$$     \frac{ 1- \|q\|} { 1+\|q\|} \leq  {\rm Re}\, \partial_s f(q) \leq \| \partial_s f(q)\| \leq \frac{ 1+ \|q\|} { 1-\|q\|}, \quad q\in B(0;1), $$
and if the first or the third inequality  becomes in  an equality for  some $q\neq 0$  then
$$ \partial_s f(q) =  (1-qe^{I\theta})^{-\ast}*(1+ qe^{I\theta} ), \quad q\in B(0;1),$$
 for some  $I\in \mathbb S$ and  $\theta\in\mathbb R$.
In particular, if  $q\in \mathbb D_{I} $  one obtains   $$\partial_s f(q) = (1+ q e^{I\theta}  ) \frac{ 1}{1-qe^{I\theta} } = (1+qe^{I\theta}) \sum_{n=0}^{\infty} q^n e^{I n \theta} ,$$
 $$=  \sum_{n=0}^{\infty} q^n e^{I n \theta} +   \sum_{n=0}^{\infty} q^{n+1} e^{I (n+1) \theta} =  1 + \sum_{n=1}^{\infty} q^n e^{In \theta} +   \sum_{n=1}^{\infty} q^{n} e^{I{n} \theta} , $$
 $$=   1 +  2\sum_{n=1}^{\infty} q^n e^{I n \theta}  , \quad q\in B(0;1).$$
 By the uniqueness of the slice regular extension, the equality $\partial_s f(q) =   1 +  2\sum_{n=1}^{\infty} q^n e^{I n \theta}$ holds on $ B(0;1)$.
Thus, if $\displaystyle f(q) =q+ \sum_{n=2}^{\infty} q^n a_n$ then
$$  1 +  \sum_{n=2}^{\infty} q^{n-1} na_n  =\partial_s f (q) =  1 +  2\sum_{n=1}^{\infty} q^n e^{I n  \theta}  ,  $$
$$  1 +  \sum_{n=1}^{\infty} q^{n} (n+1)a_(n+1)  =\partial_s f (q) =  1 +  2\sum_{n=1}^{\infty} q^n e^{I n  \theta}. $$
Therefore $\displaystyle a_{n+1} =\frac{2 e^{In \theta}}{n+1} $ for each $n\in\mathbb N$, or equivalently   $\displaystyle a_{n} =\frac{2 e^{I(n-1)  \theta}}{n} $ for each $n=2,3,4, \dots$.
\item  The proof is a direct application to the Carath\'eodory theorem for slice regular functions to $\partial_s f$ and it follows with  computations very similar to those in the previous case.
\end{enumerate}
\end{proof}

In the special case of intrinsic functions, we have:

\begin{thm}\label{deriv}
Let $f\in\mathcal{N}(B(0;1))$ and suppose that there exists  $I\in \mathbb S $  with  ${\rm Re}[\partial_{s}(f)(q)]>0$ for all $q\in \mathbb D_I$. Then ${\rm Re}[\partial_{s}(f)(q)]>0$ for all $q\in B(0 ; 1)$ and  $f$ is univalent in
$B(0 ; 1)$.
\end{thm}
\begin{proof}  From the result in the classical complex case we have that $f$ is $I$-univalent. The two statement follow from Remark \ref{rem001}, point 2 and previous theorem. Since for all $I\in \mathbb S$ one has  ${\rm Re}[\partial_s f(x+yI)] = \partial_x \alpha(x,y)  >0$ for all $x^2+y^2 <1$.
\end{proof}

\begin{ex} By the above Theorem \ref{deriv}, it follows that $f(q)=-2+2\sum_{k=1}^{\infty}\frac{q^{k}}{k}$ is an univalent function in $B(0 ; 1)$.
Indeed, it is known that
$f(z)=-2-2\log(1-z)$ is analytic in the open unit disk $\mathbb{D}_{1}$ and ${\rm Re}[f^{\prime}(z)]>0$, for all $z\in \mathbb{D}_{1}$, therefore it is univalent in $\mathbb{D}_{1}$. Also, since
$\log(1-z)=-\sum_{k=1}^{\infty}\frac{z^{k}}{k}$, we get $f(z)=-2+2\sum_{k=1}^{\infty}\frac{z^{k}}{k}$, for all $z\in \mathbb{D}_{1}$.
\end{ex}
\begin{rem} Although the above results are very simple, they seem to be very useful to produce easily many examples of injective functions in $B(0;1)$.
It is clear that an attempt to prove directly the injectivity on $B(0;1)$ of $f$ in the above example seems to be a
difficult task.
\end{rem}
\noindent
In the sequel we will consider the function $$K(q)=q\cdot [(1-q)^{2}]^{-1}$$ that corresponds to the {\it quaternionic Koebe function}. It is immediate that $K(q)=\sum_{n=1}^\infty n q^n$ since the equality holds for the restriction to a complex plane $\mathbb C_I$ and $K(q)$ is slice regular and extends (uniquely) $K(z)$, $z\in\mathbb C_I$.
\begin{ex} Let $f(q)=q+\sum_{k=2}^{\infty}q^{k} a_{k}$ be a slice regular function on $B(0;1)$, normalized, i.e. $f(0)=1-\partial_{s}(f)(0)=0$.
If $\varphi \in \mathbb{R}$ and $I\in \mathbb{S}$, then $R_{\varphi, I}(q)=e^{-I \varphi}\cdot f(e^{I \varphi} \cdot q)=q\left [1+\sum_{k=1}^{\infty}(e^{I \varphi}\cdot q)^{k}a_{k+1}\right ]$ is called a rotation of $f$. A simple reasoning shows that
if $f$ is univalent on $B(0; 1)$ then for any $\varphi \in \mathbb{R}$, $I\in \mathbb{S}$, $r\in (0, 1)$, so is $R_{\varphi, I}$.
As an application, the rotation of the quaternionic Koebe function $K(q)=q\cdot [(1-q)^{2}]^{-1}$,
$$K_{\varphi, I}(q)=e^{-I \varphi}\cdot K(e^{I \varphi} \cdot q)=q[1+\sum_{k=2}^{\infty}k (e^{I \varphi}\cdot q)^{k-1}],$$
is univalent as function of $q$ (but not regular) on the ball $B(0;1)$, because according to Theorem \ref{deriv}, the Koebe function
$f(q)=q+2q^{2}+3q^{3}+...+$ is univalent on $B(0;1)$ (see also Section 5, before Remark \ref{4.10}). Note that here for all
$a, b\in \mathbb{H}$ we have $(a\cdot b)^{k}=(a\cdot b)\cdot (a\cdot b)\cdot .... \cdot (a\cdot b)$ of $k$-times.
\end{ex}

\begin{rem}\label{geobound}
From $f\in\mathcal{N}(B(0;1))$ it is clear that in Theorem \ref{deriv} we have $\partial_{s}(f)(q)=f^{\prime}(q)$, $q\in B(0;1)$ and that $f^{\prime}\in\mathcal{N}(B(0;1))$. Considering now $f$ and $f^{\prime}$ on a slice $\mathbb{D}_{I}:=B(0;1)\bigcap \mathbb{C}_{I}$, from the complex case (see e.g. \cite{Moc}, p. 78), the condition ${\rm Re}[\partial_{s}(f)(q)]>0$ for all $q\in \mathbb D_I$ implies that $|{\rm arg}[f^{\prime}|_{I}(q)]|<\frac{\pi}{2}$, for all $q\in \mathbb D_I$. But by the definition of argument, see the Introduction, it is evident that ${\rm arg}[f^{\prime}|_{I}(q)]$ does not depend on $I\in \mathbb S$, which finally leads to the following geometric interpretation of the condition $Re[\partial_{s}(f)(q)]>0, q\in B(0;1)$ in Theorem \ref{deriv}:
$$|{\rm arg}[f^{\prime}(q)]| < \frac{\pi}{2}, \mbox{ for all } q\in B(0;1).$$
As in the complex case, a function satisfying the above inequality can be called {\it of bounded rotation}.
\end{rem}
\begin{defn} Let $I\in\mathbb S$ be a fixed element. We define the following sets
$$\mathcal{S}=\{ f\in \mathcal{R}(B(0;1)) \ \mid \ f  \ \textrm{ is univalent in } B(0;1) ,  \ f(0)=0, \  \partial_s f(0)=1\}  $$
and
$$\mathcal{S}_I = \{ f\in \mathcal{R}(B(0;1)) \ \mid \ f  \ \textrm{ is $I$-univalent in } B(0;1) ,  \ f(0)=0, \  \partial_s f(0)=1\}.  $$
\end{defn}
Note that $\mathcal{S}\subset \mathcal{S}_I$ for each $I\in \mathbb S$. The converse does not hold, as shown in Remark \ref{rem001}.
\\
As in the classical complex case, there exist some operators which preserve the set $\mathcal S$ as illustrated in the next result.

\begin{thm}\label{operations}
\begin{enumerate}\item If $0<r<1$, then $f\in \mathcal S$ if and only if $r^{-1} f(rq)\in \mathcal S$.
\item Let $u\in \mathbb H$ be such that $|u|=1$, let $f\in \mathcal S$ and $g(q) = uf(\bar u q u) \bar u$. Then $g\in\mathcal S$.
\item Let $f\in \mathcal{S}\cap \mathcal N(B(0;1))$ and $ g \in  \mathcal {R}( \ f(B(0;1) ) \ )$ an univalent function with $g(0)=0$ and $\partial_s g(0)=1$. Then $g\circ f\in \mathcal S$.
\item Let $f\in \mathcal{S}\cap \mathcal{N}(B(0,1)  )$ and $a\in \mathbb R \setminus (f(B(0;1))\cap\mathbb R)$. Then the function $g(q)= f(q)a \ast ( a - f(q))^{-*}, \quad  q\in  B(0;1)$, belongs to $\mathcal{S}\cap \mathcal{N}(B(0,1)  )$.
\item Let $f\in \mathcal{S}\cap \mathcal{N}(B(0,1)  )$, then  $g(q)= \sqrt{f(q^2)}$, belongs to $\mathcal{S}\cap \mathcal{N}(B(0,1)  )$.
\end{enumerate}
\end{thm}
\begin{proof} \ {}
\begin{enumerate}
\item The assertion immediately follows with direct computations. The function $r^{-1}f(rq)$ is called a dilation of $f$.
\item Consider  $f(q)= \sum_{n=0}^\infty  q^n a_n$ where $q\in B(0;1)$ and $(a_n)\subset\mathbb H$. Thus from the uniform convergence one  has  $$g(q)=uf (\bar u q u) \bar u=  u \sum_{n=0}^\infty  (\bar u q u)^n a_n \bar u =   \sum_{n=0}^\infty   q^n  u a_n \bar u, \quad q\in B(0;1) $$
 and so $g\in \mathcal R(B(0;1))$. Since $f\in\mathcal S$ we have $a_0=0$, $a_1=1$ and so $ua_0\bar u =0$ and $ua_1\bar u =1$. Let $q,r\in B(0;1) $ be such that   $g(q) = g(r)$. Then   $uf(\bar u q u)\bar u = uf(\bar u r u)\bar u $ or, equivalently,  $f(\bar u q u) =
f(\bar u r u)$.
 As $f\in \mathcal S$ one concludes that $q=r$, thus $g\in \mathcal S$.
\item As $f\in \mathcal N( B(0;1))$ then  $g\circ f \in \mathcal {R}(B(0;1))$ and  note that $\partial_s (g \circ f) (0)= \partial_s g (f(0)) \partial _s f(0) $. The univalence of $g\circ f$ is immediate.
\item From the usual complex case, see \cite{duren}, p. 27, point (vi), one has that $g \in \mathcal{S}_I\cap \mathcal{N}(B(0,1)  )$ for each $I\in\mathbb S$  and one concludes using  Remark \ref{rem001}, point 2.
\item Let us write $f(q)=q+a_2 q^2+ a_3q^3+\ldots$. Since $f$ vanishes just at the origin and reasoning on a fixed slice $\mathbb{C}_{I}$ as in the complex case in \cite{duren}, pp.27-28, we can choose the single-valued branch of the square root given by
    $$
    g(q)=q(1+a_2q^2+a_3q^4+\ldots)^{1/2}= q+ b_3q^3+ b_5 q^5+\ldots .
    $$
The fact that $g$ is univalent in $B(0;1)\bigcap \mathbb{C}_{I}$ can be proved as in the complex case, see \cite{duren}, p. 28. Namely, if $g(q_1)=g(q_2)$  with $q_1\not= q_2$ then $f(q_1^2)=f(q_2^2)$ but since $f$ is univalent this implies $q_1=-q_2$ and so $g(q_1)=-g(q_2)$ since $g$ is an odd function. Finally, the univalence of $g\in \mathcal{N}(B(0,1))$ on $B(0;1)$ follows from Remark \ref{rem001}, point 2.
\end{enumerate}
\end{proof}
A result ensuring univalence of a function is the following:

\begin{thm}
Let $f\in \mathcal{R} (B(0;1))$.  Assume that for any   non closed  $C^1$ curve  $\gamma$ with   $\vec\gamma\neq \vec 0$  for $t\in (0,1)$, one has that
$$   \gamma_0'(t)   \partial_{x_0} f_0 (\gamma(t)) +  \langle\   \vec{\gamma}'(t) \ , \  {\nabla} f_0 (\gamma(t)) \   \rangle_{\mathbb R^3}  > 0,  \quad \forall t\in (0,1) $$
where $f=f_0+\vec f$ and $\nabla$ is the gradient operator  in $\mathbb R^3$. Then $f$ is an univalent function in $B(0;1)$.
\end{thm}
\begin{proof}
Let $q_1, q_2 \in B(0;1) $  and   let $\gamma$ be a $C ^1$ curve in $B(0;1) $ such that $\vec \gamma (t) \neq 0$ for any $t\in (0,1)$ and  $\gamma (1) = q_2$, $\gamma(0) = q_1$. The fundamental theorem of calculus in one real variable implies
$$ f_k(q_2)- f_k(q_1) = \int_{0}^1  \sum_{m=0}^3 \gamma_m'(t) \partial_m (f_k)(\gamma (t) )  dt $$
for  $k=0,1,2,3$, where $\partial_m$ denotes $\partial_{x_m}$.  Thus
$$f(q_2)- f(q_1) = \int_{0}^1   \sum_{m=0}^3 \sum_{k=0}^3 \gamma_m'(t) \partial_m (f_k)(\gamma (t) )  e_k dt, $$
and using the linearity of the partial derivatives one has
\begin{equation}\label{equa}
f(q_2)- f(q_1) = \int_{0}^1   \sum_{m=0}^3 \gamma_m'(t) \partial_m (f)(\gamma (t) )  dt  .
\end{equation}
As $f\in \mathcal{R}(B(0;1))$, then $G[f]=0$ on $B(0;1)$,  where$$G[f] (q) = \|\vec q\|^2 \partial_0 f(q)  + \vec q \sum_{k=1}^3 q_m \partial_m f(q) =0, \quad\forall q\in B(0;1),$$
see \cite{CGS3}.
In particular,
\begin{equation}\label{equa2}
\partial_0 f (\gamma(t))=-\frac{ \vec{\gamma} (t)}{\|\vec\gamma(t)\|^2}  \sum_{k=1}^3 \gamma_m(t) \partial_m f(\gamma(t))=( \vec{\gamma} (t))^{-1}  \sum_{k=1}^3 \gamma_m(t) \partial_m f(\gamma(t)),
\end{equation}
for each  $t\in (0,1)$. As
$$\sum_{m=0}^3 \gamma_m'(t) \partial_m (f)(\gamma (t) )    = \gamma_0'(t) \partial_0 (f)(\gamma (t) )  + \sum_{m=1}^3 \gamma_m'(t) \partial_m (f)(\gamma (t) ) ,$$
and  replacing $\partial_0 (f)(\gamma(t))$ according to (\ref{equa2}), one obtains
$$\sum_{m=0}^3 \gamma_m'(t) \partial_m (f)(\gamma (t) )    =       \gamma_0'(t) (\vec{\gamma} (t))^{-1}  \sum_{m=1}^3 \gamma_m(t) \partial_m f(\gamma(t)) + \sum_{m=1}^3 \gamma_m'(t) \partial_m (f)(\gamma (t) )  $$
$$= \sum_{m=1}^3 \left[ \gamma_0'(t)( \vec{\gamma} (t))^{-1}   \gamma_m(t) +  \gamma_m'(t) \right] \partial_m (f)(\gamma (t) )     , \quad \forall t\in (0,1).$$
The real part of (\ref{equa}) equals
$${\rm Re}(f(q_2)- f(q_1) ) = {\rm Re}\left(  \int_0^1 \sum_{m=0}^3 \gamma_m'(t) \partial_m (f)(\gamma (t) )    dt \right) $$
$${\rm Re}(f(q_2)- f(q_1) ) = {\rm Re}\left(  \int_0^1  \sum_{k=1}^3 \left[ \gamma_0'(t)( \vec{\gamma} (t))^{-1}   \gamma_m(t) +  \gamma_m'(t) \right] \partial_m (f)(\gamma (t) )   dt \right) $$
$$=   \int_0^1 \left(    -\gamma_0'(t)        \langle \  \vec{\gamma} (t)^{-1}  \ , \  \sum_{k=1}^3  \gamma_m(t) \partial_m (\vec{f})(\gamma (t) )  \  \rangle  +      \sum_{k=1}^3  \gamma_m'(t)  \partial_mf_0(\gamma (t) )  \right)  dt $$
Multiplying  both sides of (\ref{equa2})  by $\vec{\gamma}(t)$  and considering its  vectorial  part, we have
$$       \partial_0 f_0 (\gamma(t))     \vec{\gamma} (t)     +   \left[    \vec{\gamma} (t)   ;   \partial_0 \vec{f}(\gamma (t)) \right] = \sum_{k=1}^3  \gamma_m(t) \partial_m (\vec{f})(\gamma (t) )  .  $$
 Therefore
$${\rm Re}(f(q_2)- f(q_1) )=   \int_0^1 \left(    -\gamma_0'(t)        \langle \  \vec{\gamma} (t)^{-1}  \ , \   \partial_0 f_0 (\gamma(t))     \vec{\gamma} (t)   \  \rangle  +      \sum_{k=1}^3  \gamma_m'(t)  \partial_m f_0(\gamma (t) )  \right)  dt $$
or, equivalently, $${\rm Re}(f(q_2)- f(q_1) )=   \int_0^1 \left(    \gamma_0'(t)          \partial_0 f_0 (\gamma(t))   +      \sum_{k=1}^3  \gamma_m'(t)  \partial_m f_0(\gamma (t) )  \right)  dt $$
$${\rm Re}(f(q_2)- f(q_1) )=   \int_0^1 \left(    \gamma_0'(t)          \partial_0 f_0 (\gamma(t))   +     \langle \  \vec{\gamma}'(t) \  , \  \vec{\Delta}  f_0(\gamma (t) ) \  \rangle  \right)  dt $$
Then
${\rm Re}(f(q_2)- f(q_1) ) >0$ and thus $f$ is univalent.
\end{proof}

 \begin{thm}\label{geom} Let $f\in \mathcal{S}\cap \mathcal N( B(0;1))$. We have:
\begin{enumerate}
\item $B(0;\frac{1}{4})\subset f(B(0;1))  ; $
\item
$$ \frac{1-r}{(1+r)^3} \le \|\partial_s f(q)\| \le \frac{1+r}{(1-r)^3} ,\quad    0<\|q\| = r < 1 ;$$
\item
$$ \frac{r}{(1+r)^2} \le \| f(q)\|  \le \frac{r}{(1-r)^2} ,\quad    0<\|q\| = r < 1 ;$$
\item
$$ \frac{1-r}{(1+r) } \le \| q \partial_s f(q)   \ast f(q) ^{-*}\|  \le \frac{1+r}{(1-r) } ,\quad    0<\|q\| = r < 1 ;$$
\item For each $I\in \mathbb S$,
$$r\int_{0}^{2\pi}   \|\partial_s f(re^{I\theta})\| d\theta \leq \frac{2\pi r(1+r)}{(1-r)^2}.$$
\end{enumerate}
\end{thm}
\begin{proof}As $f\in \mathcal N(B(0;1))$, since $f[B(0;1)\bigcap \mathbb{C}_{I}]\subset \mathbb{C}_{I}$ for every $I\in {\mathcal{S}}$, we can use on each slice the usual complex theory. In fact, most of them are metrical properties which do not depend on the choice of $I$.

1)  Apply  Theorem 2.3 (Koebe One-Quarter Theorem) in e.g. \cite{duren}. Note that since the quaternionic Koebe functions $K_{\theta}(q)=q\cdot [(1+qe^{I_{q}\theta})^{2}]^{-1}$, $\theta\in \mathbb{R}$ are not, in general, intrinsic functions (except the cases $\theta=p \pi$, $p\in \mathbb{Z}$), we have only $B(0;1/4)\subseteq \bigcap_{f\in \mathcal{S}\cap \mathcal N( B(0;1))}f(B(0;1))$ and not equality, as in the complex case.

2)  Apply Theorem 2.5 (Distortion Theorem) in e.g. \cite{duren} on each slice. As at the above point 2), the only intrinsic quternionic Koebe functions are
$K_{0}(q)=q\cdot [(1+q)^{2}]^{-1}$ and $K_{\pi}(q)=q\cdot [(1-q)^{2}]^{-1}$, which attaint equalities in the obtained inequalities only at the points $q=r$ and $q=-r$. For this reason, we do not have the situation in the complex case, when the equalities are attained if and only if $f$ is a suitable chosen complex Koebe function.

3) Apply Theorem 2.6 (Growth Theorem) in e.g. \cite{duren} on each slice and the remark on the Koebe's functions from the point 2.

4)  Apply Theorem 2.7 in e.g. \cite{duren} on each slice and again the remark on the Koebe's functions from the point 2.

5) Apply Theorem 2.9 in e.g. \cite{duren} on each slice.
\end{proof}
According to Definition 2.2 in \cite{CGS2}, we define another subset of slice regular functions:
$$\mathcal V_{I}( B(0;1))=\{f\in \mathcal{R}(B(0;1)) ; f(B(0;1)\cap \mathbb{C}_{I})\subset \mathbb{C}_{I}\}.$$
Therefore, we introduce the set
$$\mathcal{V}(B(0;1))=\bigcup_{I\in \mathcal{S}}\mathcal V_{I}( B(0;1)).$$
It is clear that as $\mathcal N( B(0;1))=\bigcap_{I\in \mathcal{S}}\mathcal V_{I}( B(0;1))$, the set $\mathcal{V}(B(0;1))$ is much larger than $\mathcal N( B(0;1))$.
According to Corollary 2.8 in \cite{CGS2}, we have $f\in \mathcal V_{I}( B(0;1))$ if and only if for all $q\in B(0;1)$,
$f(q)=\sum_{k=0}^{\infty}q^{k}a_{k}$, with $a_{k}\in \mathbb{C}_{I}$ for all $k=0, 1, ..., .$
\\
We have the following result:
\begin{thm}
Let $f\in \mathcal{S}\cap \mathcal V( B(0;1))$. We have :
\begin{enumerate}
\item If $f\in\mathcal V_I(B(0;1))$ and $f(B(0;1))   \   \subset \  \mathbb H \setminus \overline{B(0;1)}  $ is of the form
$$  f(q)=q +  \sum_{n=0}^{\infty } q^{-n} a_n, \quad q\in B(0;1),  $$
then the area of $f(B(0;1)\cap \mathbb C_I)$ is
$$  2\pi^2 (1- \sum_{n=1} n|a_n|^2) ; $$
\item (Bieberbach- de Branges Theorem) If $$f(q) = q + \sum_{n=2}^{\infty} q^n a_n, \quad q\in  B(0;1),$$
then  $|a_n| \le n$ for each $n=2,3,4, \dots$.
\end{enumerate}
\end{thm}
\begin{proof}
1)  The function $f$, by assumption, has coefficients in $\mathbb C_I$. Theorem 2.1 (Area Theorem) in   \cite{duren}, gives that the area  $A(f(\mathbb B \cap \mathbb C_I) )$ is
$$A(f(\mathbb B \cap \mathbb C_I) ) =\pi (1- \sum_{n=1} n|a_n|^2).$$

2) Since the coefficients of $f$ belongs to a given complex plane $\mathbb C_I$, according to the result in e.g. \cite{DeBrange} applied on a chosen slice, we get the corresponding estimates for the coefficients of
the series development of $f$.
\end{proof}

\begin{rem}\label{inter}
It is worth noting that for any fixed $I\in \mathcal{S}$, the set $\mathcal V_{I}( B(0;1))$ can be
generated from the whole class of univalent analytic functions of complex
variable in the unit disk, by simply replacing in each such complex
function $z$ by $q$, and in the coefficients of the series development, the
small $i$ with the capital $I$. In this way, any univalent analytic function of complex variable in the unit disk with at least one non-real coefficient in its series development, generates an infinity of quaternionic functions in $\mathcal{V}(B(0;1))$.
\end{rem}

\noindent{\bf Open Question.} It is a natural question if $\mathcal{V}(B(0;1))$ is the largest class of univalent slice regular functions in which the Bieberbach - de Branges result holds.  There could exists $f\in \mathcal{S}$ which is not in $\mathcal{V}(B(0;1))$, for which this result does not hold.

\noindent {\bf Algebraic and Geometric Interpretations.} Some results of Theorem \ref{geom}, can be interpreted as geometric or algebraic properties of infinite differentiable injective transformations from $\mathbb{R}^{4}$ to $\mathbb{R}^{4}$. Indeed, any $f:B(0,1)\to \mathbb{H}$ can be written in the form
\begin{equation}\label{geo}
f(x_{1}+ix_{2}+jx_{3}+kx_{4})
\end{equation}
$$=P_{1}(x_1, x_2, x_3, x_4)+i P_{2}(x_1, x_2, x_3, x_4)+j P_{3}(x_1, x_2, x_3, x_4)
+k P_{4}(x_1, x_2, x_3, x_4),$$
with all $P_{k}(x_1, x_2, x_3, x_4)$, $k=1, 2, 3, 4$, real-valued, and therefore $f$ can be also be viewed as
the transformation
$$(x_1, x_2, x_3, x_4) \to$$
$$ (P_{1}(x_1, x_2, x_3, x_4), P_{2}(x_1, x_2, x_3, x_4), P_{3}(x_1, x_2, x_3, x_4),
P_{4}(x_1, x_2, x_3, x_4)).$$
Below we present a few illustrations.

\noindent 1) In the case of Theorem \ref{geom}, 2), a first geometric interpretation is the relationship
$$B(0;1/4)\subseteq \bigcap_{f\in \mathcal{S}\cap \mathcal N( B(0;1))}f[B(0;1)].$$
Note that due to the fact that most of the quaternionic Koebe functions are not intrinsic functions, the situation is different from the complex case, when above instead of $\subseteq$ we have equality.

Secondly, by identifying $q=x_{1}+ix_{2}+jx_{3}+kx_{4}\in \mathbb{H}$ with $(x_{1}, x_{2}, x_{3}, x_{4})$ and defining the Euclidean distance in $\mathbb{R}^{4}$ by $d_{\mathbb{R}^{4}}^{(E)}(0, q)=\sqrt{x_{1}^{2}+x_{2}^{2}+x_{3}^{2}+x_{4}^{2}}$, the inclusion in Theorem \ref{geom}, 2), can be written as follows : if $f\in \mathcal{S}\cap \mathcal N( B(0;1))$ then for any $r\in (0, 1/4)$, there exists $q$ with $d_{\mathbb{R}^{4}}^{(E)}(0, q)<1$ solution of the equation $d_{\mathbb{R}^{4}}^{(E)}(0, f(q))=r$.
For example, if we consider $f(q)=q+\frac{q^{2}}{4}$ which is in $\mathcal{S}\cap \mathcal N( B(0;1))$, since by simple calculation we get
$f=P_{1}+iP_{2}+jP_{3}+kP_{4}$ where $P_{1}=x_{1}+\frac{1}{4}\left (x_{1}^{2}-x_{2}^{2}-x_{3}^{2}-x_{4}^{2}\right )$, $P_{2}=x_{2}+\frac{x_{1} x_{2}}{2}$, $P_{3}=x_{3}+\frac{x_{1}x_{3}}{2}$, $P_{4}=x_{4}+\frac{x_{1}x_{4}}{2}$, it follows that for any $r\in (0, 1/4)$, the algebraic equation
$P_{1}^{2}+P_{2}^{2}+P_{3}^{2}+P_{4}^{2}=r^{2}$ has at least one solution $(x_{1}, x_{2}, x_{3}, x_{4})$ with $x_{1}^{2}+x_{2}^{2}+x_{3}^{3}+x_{4}^{2}<1$. It is clear that the attempt to solve this algebraic equation by direct methods  is not an easy task.

\noindent 2) In the case of Theorem \ref{geom}, 3), denoting by $R_{1}=\frac{r}{(1+r)^{2}}$ and $R_{2}=\frac{r}{(1-r)^{2}}$, a geometric interpretation of the inequalities is as follows :
$$B(0;R_{1})\subseteq \bigcap_{f\in \mathcal{S}\cap \mathcal N( B(0;1))}f[B(0;r)], \, \, \, \, \, \bigcup_{f\in \mathcal{S}\cap \mathcal N( B(0;1))}f[B(0;r)]\subseteq B(0;R_{2}).$$
Note again that due to the fact that most of the quaternionic Koebe functions are not intrinsic functions, the situation is different from the complex case, when above instead of $\subseteq$ we have equalities.

Also, with the notations from the above point, the inequalities in Theorem \ref{geom}, 3), can be written as follows :
for any $q$ with $0<d_{\mathbb{R}^{4}}^{(E)}(0, q)<1$ we have
$$\frac{1-d_{\mathbb{R}^{4}}^{(E)}(0, q)}{(1+d_{\mathbb{R}^{4}}^{(E)}(0, q))^{3}}\le d_{\mathbb{R}^{4}}^{(E)}(0, \partial_{s}f(q))\le \frac{1+d_{\mathbb{R}^{4}}^{(E)}(0, q)}{(1-d_{\mathbb{R}^{4}}^{(E)}(0, q))^{3}}.$$
Similarly, the inequalities in Theorem \ref{geom}, 4) can be written as
$$\frac{d_{\mathbb{R}^{4}}^{(E)}(0, q)}{(1+d_{\mathbb{R}^{4}}^{(E)}(0, q))^{2}}\le d_{\mathbb{R}^{4}}^{(E)}(0, f(q))\le \frac{d_{\mathbb{R}^{4}}^{(E)}(0, q)}{(1-d_{\mathbb{R}^{4}}^{(E)}(0, q))^{2}},$$
for any $q$ with $0<d_{\mathbb{R}^{4}}^{(E)}(0, q)<1$.

For particular functions $f$ as, for example, $f(q)=q+\frac{q^2}{4}$, these relationships become algebraic inequalities which are not easy
to be proved by direct methods.
%
%
%
%

\begin{thm}\label{area}(Area Theorem)
Let $f$ be a slice regular function on $\Delta:=\{ q\in \mathbb H \ \mid \ \|q\|>1 \}$, such that $$ f(q) = q+ \sum_{n=0}^{\infty} q^{-n} a_n.$$

 Assume that for some $I\in \mathbb S$ and $J\in\mathbb S$ orthogonal to $I$, one has $f_{|\mathbb D_I}=f_1+f_2J$ where $f_1(z)$ and $z+f_2(z)$ are univalent functions. Then
\begin{enumerate}
\item
$$\left(   \left[ \frac{1}{2} (f - If I)(\Delta_I) \right]^{c}\,  \right)  +m\left(    \left[ \left(  \mathcal I+ \frac{1}{2}( f  +IfI)
{\bar J}
\right)  (\Delta_I)   \right]^{c} \,  \right)      = \pi (2-\sum_{n=1}^{\infty}  n \|a_n\|^2 ).$$
where $ \Delta_I = \Delta \cap \mathbb C_I$, {}  $\mathcal I$ is the  identity function on $B(0;1)$ and $m$ is the  two dimensional outer measure.
Here $A^{c}$ denotes $\mathbb{H}\setminus A$.

\item We have the inequality
$$ \sum_{n=1}^{\infty}  n \|a_n\|^2  \leq 2 .$$

\item We have $\|a_1\|\leq \sqrt{2}$, with equality if and only if $f(q) = q+a_0 + a_1 q^{-1}$.
\end{enumerate}
\end{thm}
\begin{proof}
(1) Let $I,J\in\mathbb S$ with $J\perp I$ and let $f_1,f_2 $ be holomorphic functions on $\Delta_I$  such that $f\mid_{\Delta_I} = f_1+f_2J$ such that $f_1(z)$ and $z+f_2(z)$ are univalent functions. Moreover, if $a_n = a_{1,n} + a_{2,n} J$ with $a_{k,n}\in \mathbb C_I$ for each $k=1,2$ and any $n\in \mathbb N\cup \{0\}$, then   $$ f_1(z) = z+ \sum_{n=0}^{\infty} z^{-n} a_{1,n},  \quad  z + f_2(z)=  z  + \sum_{n=0}^\infty  z^{-n}a_{2,n} , $$ on  $\Delta_I$.  Thus, since the above functions  $f_1$, $\mathcal I+f_2$ have a simple pole at infinity with residue 1 (where  $\mathcal I$ represents the  identity function on $B(0;1)$),  from Theorem 1.2, p. 29 in \cite{duren} and observing that $f_1= \frac{1}{2}(f - IfI ) \mid_{\Delta_{I}}$,  {}   $f_2=\frac{1}{2}(f+ IfI)\bar J \mid_{\Delta_I} $,  {}    $a_{1,n} =\frac{1}{2}(a_n - Ia_n I) $ and
$a_{2,n }= \frac{1}{2}(a_n + Ia_nI)\bar {J}$,  for each $n\in \mathbb N \cup \{0\}$,  one has that
$$m(   \left[ \frac{1}{2} (f - If I)(\Delta_I) \right]^{c} \ ) =  \pi (1- \frac{1}{4}\sum_{n=1}^{\infty}  n\|a_n - Ia_n I\|^2),  $$
$$  m\left(    \left[ \left(  \mathcal I+ \frac{1}{2}( f  +IfI)\bar J   \right)  (\Delta_I)   \right]^{c}  \  \right)      =  \pi(1-\frac{1}{4}\sum_{n=1}^{\infty}  n\|a_n + Ia_nI\|^2).  $$
  Finally,
from the previous formula and the parallelogram identity: $$\|a_n - Ia_nI\|^2  + \|a_n + Ia_nI\|^2 = 2 \|a_n\|^2 + 2\| Ia_nI\|^2 =4 \|a_n\|^2,$$ for each $n\in \mathbb N$, one concludes
$$m(   \left[ \frac{1}{2} (f - If I)(\Delta_I) \right]^{c} \ )  +m\left(    \left[ \left(  \mathcal I+ \frac{1}{2}( f  +IfI) \bar J  \right)  (\Delta_I)   \right]^{c}  \  \right)         = \pi (2-\sum_{n=1}^{\infty}  n \|a_n\|^2 ).$$
\\ Point (2) follows from (1) and (3) is an immediate consequence of (2).
\end{proof}
A variation of the previous result is the following:
\begin{thm}
Let $f\in \mathcal{R}(B(0;1))$, $f(q) =  \sum_{n=0}^{\infty }    q^n a_n$ and let $f_1,f_2 \in Hol(\mathbb D_I)$ be such that
$f\mid_{D_I} =f_1+f_2J$. Assume that the functions $z + f_1(z^{-1}), \quad z+ f_2(z^{-1})$ are  univalent functions on $\Delta_I :=\{  z \in \mathbb C_I \ \mid \  |z|>1 \}$.
Then   $$ \sum_{n=1}^{\infty} n\|a_n\|^2\leq 2$$

\end{thm}
\begin{proof}
From our assumptions it follows that
   $$f\mid_{\mathbb D_I}(z^{-1}) = \sum_{n=0}^{\infty }    z^{-n} a_n,  \quad z\in \Delta_I$$
and if  $a_n = a_{1,n}+ a_{2,n} J$ with $a_{1,n}, a_{2,n}\in \mathbb C_I$ for each $n\in \mathbb N\cup\{0\}$, then
$$ z+f_1({z^{-1}} ) = z +  \sum_{n=0}^{\infty }    z^{-n} a_{1,n}, \quad   z+f_2({z^{-1}} ) = z +  \sum_{n=0}^{\infty }    z^{-n} a_{2,n}. $$

Thus repeating  the computations in proof of Theorem \ref{area} one obtains  that
$$ \sum_{n=1}^{\infty} n\|a_n\|^2\leq 2.$$
\end{proof}
\begin{rem}{\rm
As a direct consequence of previous result one has that
$$ \|a_1\| \leq \sqrt 2 ,  \quad  \|a_2\| \leq 1. \quad $$
Note that   if $a_1=1$ then
$$ \sum_{n=2}^{\infty} n\|a_n\|^2\leq 1.$$
}
\end{rem}
For  $t\in (-1,1)$  denote
$$  T_t (q)= \frac{ q+t}{ 1+tq}, \quad q\in B(0; 1) . $$
We have the following result:
 \begin{prop}
 Let $f\in \mathcal{R}(B(0;1))$. Let $I,J\in \mathbb S$, $J$ orthogonal to $I$,  and   $f\mid_{\mathbb D_I} =f_1+f_2J$  with  $f_1,f_2\in Hol(\mathbb D_I)$. If there exists $0<\delta<1$ such that
   $$  z  +  f_1 \circ T_t (z^{-1}) , \quad   z  +  f_2 \circ T_t (z^{-1}) , \quad z\in \Delta_I  $$
 are univalent functions  on $\Delta_I$   for each  $t\in (-\delta, \delta)$, then
$$ \|  \partial_s f(t)  \|\leq \frac{\sqrt{2}}{1-t^2}  $$
and
  $$\|-2 t \partial_s f(t) +  (1-t^2) \partial_s^2 f(t)    \|\leq  \frac{1}{1-t^2}  $$
for each $t\in(-\delta, \delta)$.
\end{prop}
\begin{proof}
As  $g = f \circ T_t  \in  \mathcal {R}(B(0;1))$  and it satisfies the conditions of previous proposition, then denoting
$$g \circ T_t^{-1} (q) = \sum_{n=0}^{\infty} q^n a_n,$$
one has that   $\|a_1\|^2 \leq  2$ and $\|a_2\|^2\leq  1$. Note also that
 $$ a_1=\partial_s g  (0)= \partial_sT_t(0) \partial_s f(T_t(0))  =  (1-t^2) \partial_s f(t)      $$
and
 $$ a_2=\frac{1}{2}\partial_s^2 g  (0)= \partial_s^2T_t(0) \partial_s f(T_t(0))   + (\partial_sT_t(0) )^2\partial_s^2 f(T_t(0))  =$$
$$= -2 (1-t^2)t \partial_s f(t) +  (1-t^2)^2 \partial_s^2 f(t)      $$
Therefore
$$ \| (1-t^2) \partial_s f(t)  \|^2\leq 2  $$
 and
$$\|-2 (1-t^2)t \partial_s f(t) +  (1-t^2)^2 \partial_s^2 f(t)    \|^2\leq 1   $$
for each $t\in (-\delta,\delta)$.
\end{proof}


\section{Subordination of Slice Regular Functions}
In this section firstly we introduce and study a notion of composition of slice regular functions which is of independent interest, but which also will be used to study the subordination concept between two slice regular functions.
As it is well know, see for example \cite{CS}, the composition  $f\circ g$ of two slice regular functions is not anymore slice regular, unless $g$ belongs to the subclass $\mathcal N$, i.e., it is quaternionic intrinsic.
However, in order to define the notion of subordination it is necessary to have a notion of composition between slice regular functions. To motivate our choice of the notion of composition, we also recall that the pointwise product does not preserve slice regularity (unless one of the functions belongs to $\mathcal N$) and while the slice regularity is preserved by the so-called $*$-product. The power of a function is slice regular only if it is computed with respect to the slice product and we will write $(w(q))^{*n}$ when we take the $n$-th power with respect to the $*$-product.
This justifies the definition below. We first treat the case of formal power series.
\begin{defn}\label{bullet}
Denoting $g(q)=\sum_{n=0}^{\infty} q^{n}a_{n}$ and $w(q)=\sum_{n=1}^{\infty}q^{n} b_{n}$.  We define
$$(g \bullet w)(q)=\sum_{n=0}^\infty (w(q))^{*n}a_n.$$
\end{defn}
\begin{rem}{\rm Note that if $w \in {\mathcal{N}}(B(0;1))$, then $g\bullet w$ becomes $g\circ w$ where $\circ$ represents the usual composition of two functions.
Note that if $w$ is quaternionic intrinsic $(w(q))^{*n}=(w(q))^n$ so, in particular, $q^{*n}=q^n$.
}
\end{rem}

\begin{rem}\label{cartanrem}{\rm Following \cite{cartan} we call {\em order} of a series $f(q)=\sum_{n=0}^{\infty} q^{n}a_{n}$ and we denote it by $\omega(f)$, the least integer $n$ such that $a_n\not=0$ (with the convention that  the order of the series identically equal to zero is $+\infty$). Assume to have a family $\{f_i\}_{i\in \mathcal I}$ of power series where $\mathcal I$ is a set of indices. The family is said to be summable if for any $k\in\mathbb N$, $\omega(f_i)\geq k$ for all except a finite number of indices $i$. By definition the sum
of $\{f_i\}$ where $f_i(q)=\sum_{n=0}^{\infty} q^{n}a_{i,n}$ is
$$
f(q)=\sum_{n=0}^{\infty} q^{n}a_{n},
$$
where $a_n=\sum_{i\in\mathcal I}a_{i,n}$. Note that the definition of $a_{n}$ makes sense since our hypothesis guarantees that for any $n$ just a finite number of $a_{i,n}$ are nonzero.
}
\end{rem}
\begin{rem}
The hypothesis $b_0=0$ in Definition \ref{bullet} is necessary in order to guarantee that in the term $(w(q))^{*n}$ the minimum power of $q$ is at least $q^n$ or, in other words, that $\omega(w(q)^{*n})\geq n$ (for all indices). In this way, the series $\sum_{n=0}^\infty (w(q))^{*n}a_n$ is summable, see Remark \ref{cartanrem}, and we can regroup the powers of $q$.
\end{rem}
This composition,  in general it is not associative as one can directly verify with an example:
by taking $f(q)=q^2 c$, $g(q)=qa$ and $w(q)=q^2b$ we have
$((f\bullet g) \bullet w)=q^4 b^2a^2c$ while $(f\bullet (g \bullet w))(q)=q^4babac$. However, we will prove that the composition is associative in some cases and to this end we need a preliminary Lemma.
\begin{lem}\label{propertybullet}
Let $f_1(q)=\sum_{n=0}^{\infty} q^{n}a_{n}$, $f_2(q)=\sum_{n=0}^{\infty} q^{n}b_{n}$, and $g(q)=\sum_{n=1}^{\infty} q^{n}c_{n}$. Then:
\begin{enumerate}
\item $(f_1+f_2)\bullet g= f_1\bullet g+f_2\bullet g$;
\item if $g$ has real coefficients $(f_1*f_2)\bullet g=(f_1 \bullet g)*(f_2 \bullet g)$;
\item if $g$ has real coefficients $f^{*n}\bullet g=(f \bullet g)^{*n}$.
\end{enumerate}
Moreover, if $\{f_i\}_{i\in\mathcal I}$ is a summable family of power series then $\{f_i\bullet w\}_{i\in\mathcal I}$ is summable and
\begin{enumerate}
\item $(\sum_{i} f_i) \bullet w=\sum_i (f_i\bullet w).$
\end{enumerate}
\end{lem}
\begin{proof}
To prove $(1)$ observe that
$$
(f_1+f_2)\bullet g=
\sum_{n=0}^\infty g^{*n}(a_{n}+b_n)=\sum_{n=0}^\infty g^{*n}a_{n}+\sum_{n=0}^\infty g^{*n}b_n= f_1\bullet g+f_2\bullet g.
$$
Let us  prove $(2)$.  We have $f_1*f_2(q)= \sum_{n=0}^\infty q^{n}(\sum_{r=0}^n a_rb_{n-r})$ and so
$$
((f_1*f_2)\bullet g)(q)=\sum_{n=0}^\infty (g(q))^{n}(\sum_{r=0}^n a_rb_{n-r})
$$
and, taking into account that the coefficients of $g$ are real:
$$
(f_1 \bullet g)(q)*(f_2 \bullet g)(q)=\left(\sum_{n=0}^\infty(g(q)^{n} a_n\right)*\left(\sum_{m=0}^\infty g(q)^{m} b_m\right)= \sum_{n=0}^\infty g(q)^{n}  (\sum_{r=0}^n a_rb_{n-r}).
$$
We prove $(3)$ by induction. Observe that the statement is true for $n=2$ since it follows from $(2)$. Assume that the assertion is true for the $n$-th power. Let us show that it holds for $n+1$.
Let us compute
\[
(f^{*(n+1)}\bullet g)(q)=( (f^{*(n)}*f)\bullet g)(q)\overset{(2)}=(f \bullet g)^{*n}*(f\bullet g)=(f \bullet g)^{*(n+1)},
\]
and the statement follows.\\
Finally, to show $(4)$ we follow \cite[p. 13]{cartan}. Let $f_i(q)=\sum_{n=0}^\infty q^n a_{i,n}$ so that we have, by definition
$$
\sum_{i\in\mathcal I} f_i(q)= \sum_{n=0}^\infty q^n (\sum_{i\in\mathcal I}a_{i,n}).
$$
Thus we obtain
\begin{equation}\label{e1}
(\sum_{i\in\mathcal I} f_i(q))\bullet g = \sum_{n=0}^\infty g(q)^{*n} (\sum_{i\in\mathcal I}a_{i,n})
\end{equation}
and
\begin{equation}\label{e2}
\sum_{i\in\mathcal I} (f_i\bullet g)(q)= \sum_{i\in\mathcal I} (\sum_{n=0}^\infty  g(q)^{*n} a_{i,n}).
\end{equation}
Now observe that, by hypothesis on the summability of $\{f_i\}$, each power of $q$ involves just a finite number of the coefficients $a_{i,n}$ so we can apply the associativity of the addition in $\mathbb H$ and
so (\ref{e1}) and (\ref{e2}) are equal.
\end{proof}
\begin{prop}\label{assoc}
If  $f(q)=\sum_{n=0}^{\infty} q^{n}c_{n}$,
$g(q)=\sum_{n=1}^{\infty} q^{n}a_{n}$, $w(q)=\sum_{n=1}^{\infty}q^{n} b_{n}$ and $w$ has real coefficients, then
$
(f\bullet g) \bullet w = f\bullet (g \bullet w).
$
\end{prop}

\begin{proof}
We adapt the proof of Proposition 4.1 in \cite{cartan}.
Let us begin by proving the equality in the special case in which $f(q)=q^n a_n$.
We have:
$$
((f\bullet g) \bullet w )(q)=(g^{*n}\bullet w)a_n
$$
and
$$
(f\bullet (g \bullet w))(q)= (g \bullet w)^{*n}(q) a_n.
$$
Lemma \ref{propertybullet}, $(3)$ shows that $(g^{*n}\bullet w)=(g^{*n}\bullet w)$ and the equality follows.
The general case follows by considering $f$ as the sum of the summable family $\{q^na_n\}$ and using the first part of the proof:
\[
(f\bullet g)\bullet w =\sum_{n=0}^\infty (g^{*n}\bullet w) a_n =\sum_{n=0}^\infty (g\bullet w)^{*n} a_n= f\bullet (g\bullet w).
\]
 \end{proof}
 So far, we have considered power series without specifying their set of convergence. We now take care of this aspect, by proving the following result which is classical for power series with coefficients in a commutative ring, see \cite{cartan}.

 \begin{prop} Let $g(q)= \sum_{n=0}^\infty q^n a_n$ and $f(q)= \sum_{n=1}^\infty q^n b_n$ be convergent in the balls of nonzero radius $R$ and $\rho$, respectively, and let $h(q)=(g\bullet f)(q)$. Then the radius of convergence of $h$ is nonzero and if $r>0$ is any real number such that $\sum_{n=1}^\infty  r^n \| b_n\|<R$, then the radius of convergence of $h$ is greater than or equal $r$.
\end{prop}
\begin{proof}
First of all, let us observe that
$$
\left\| \left(\sum_{m=1}^\infty q^m b_m\right) ^{*n}\right\|\leq  \left(\sum_{m=1}^\infty \| q\|^m \| b_m \|\right)^n.
$$
In fact, we have
\begin{equation}\label{recursion}
\begin{split}
\left\| \left(\sum_{m=1}^\infty q^m b_m\right)*\left(\sum_{m=1}^\infty q^m c_m\right) \right\| &=\| \sum_{m_1}^\infty q^{m}(\sum_{r=0}^m b_r c_{m-r}) \|\\
&\leq \sum_{m_1}^\infty \|q\|^{m}(\sum_{r=0}^m \|b_r\| \| c_{m-r}\| \\
&=(\sum_{m=1}^\infty \| q\|^m \|b_m\|)(\sum_{m=1}^\infty \| q\|^m \|c_m\|),
\end{split}
\end{equation}
and so the statement is true for $n=2$ the rest follows recursively by using (\ref{recursion}).
So we have
\begin{equation}\label{norms}
\begin{split}
\left\| \sum_{n=0}^\infty \left(\sum_{m=1}^\infty q^m b_m\right)^{*n} a_n\right\| &\leq \sum_{n=0}^\infty  \left\| \left(\sum_{m=1}^\infty q^m b_m\right) ^{*n}\right\|  \|a_n\| \\
&\leq \sum_{n=0}^\infty   \left(\sum_{m=1}^\infty \|q\|^m \|b_m\|\right)^{n}  \|a_n\|
\end{split}
\end{equation}
Now, since the series expressing $f$ is converging on a ball of finite radius, there exists a positive number $r$ such that $\sum_{n=1}^\infty r^n \|b_n\| $ is finite.
Moreover, $\sum_{n=1}^\infty r^n \|b_n\|=r\sum_{n=1}^\infty r^{n-1} \|b_n\|\to 0$ for $r\to 0$ and so there exists $r$ such that $\sum_{n=1}^\infty r^n \|b_n\| <R$.
Thus, from (\ref{norms}), we have that
$$
\sum_{n=0}^\infty   \left(\sum_{m=1}^\infty  r^m \|b_m\|\right)^{n}  \|a_n\| =\sum_{n=0}^\infty   \sum_{m=1}^\infty  r^m \gamma_m <\infty .
$$
Thus we have that $(g\bullet f)(q)=\sum_{m=0}^\infty q^m c_m$ and $\| c_m\|\leq \gamma_m$ and thus the radius of convergence of $g\bullet f$ is at least equal to $r$, where $r$ is a sufficiently small positive real number.
\end{proof}

We are now ready to give the notion of subordination:

\begin{defn}\label{defsub}Let $f, g$ be (left) slice regular on $B(0;1)$.
\begin{enumerate}
\item  We say that $f$ is subordinated to $g$ and we write $f \prec g$ if there
exists $w\in {\mathcal{R}}(B(0;1))$ with $w(0)=0$ and $\|w(q)\|<1$, for all $q\in B(0;1)$ such that $f(q)=(g \bullet w)(q)$ for all $q\in B(0;r)$, where $r$ is a suitable number in $(0,1)$.
\item  We write $f \prec_{\mathcal N} g$ if there
exists $w\in {\mathcal{N}}(B(0;1))$ with $w(0)=0$ and $\|w(q)\|<1$, for all $q\in B(0;1)$ such that $f(q)=(g \bullet w)(q)$ for all $q\in B(0;1)$.
\item  Let $I\in \mathbb S$, then  $f$ is $I$-subordinated to $g$ and we write $f \prec_I g$ if there
exists $ w_I \in Hol (\mathbb D_I)$ such that  $|w_I(z)| \leq |z| $ for all $z\in \mathbb D_I$ and  $f\mid_{\mathbb D_I}(z)=(g \mid_{\mathbb D_I} \circ  w_I)(z)$ for all $z\in \mathbb D_I$.
\end{enumerate}
Note that if  $f \prec_{{\mathcal{N}}} g$ then $f \prec_I g$ for each $I\in \mathbb S$, but the converse is false, for example: $f(q)= qI, \quad g(q)= q $ for each $q\in B(0;1)$ and $w_I(z) = zI$ for $z\in \mathbb D_I$
\end{defn}
\begin{thm}  Let $f, g\in\mathcal R(B(0;1))$ and let $I\in \mathbb S$ such that  $f \prec_{I} g$.
\begin{enumerate}
\item Then $\|\partial_s f(0)\|\leq \|\partial_s g(0)\|$.
\item If $g_1,g_2,g_3,g_4 \in \mathcal{N}(B(0;1) ) $ are such that $g=g_1+g_2I+g_3J+g_4IJ$, with $J\in\mathbb S$ and $J\perp I$, then for any $q\in B(0;1)$ there exist $q_1,q_2 \in B(0;1)$  satisfying
$$ f(q) = g_1(q_1)+g_2 (q_1)I+g_3(q_2)J+g_4(q_2)IJ.$$
\end{enumerate}
\end{thm}
\begin{proof} \ {}
\begin{enumerate}
\item Set  $f\mid_{\mathbb D_I }= v_1+ v_2J$ and  $g\mid_{\mathbb D_I }= h_1+ h_2J$ where $v_1, v,_2, h_1,h_2 \in Hol (\mathbb D_I )$  and $J\in \mathbb S $, $J\perp I $. From the complex case, see \cite{duren},  one has $$\|\partial_s f(0)\|^2 =  \|\partial_s f \mid_{\mathbb D_I}(0)\|^2 =  \|\partial_s v_1 (0)\|^2  + \|\partial_s v_2 (0)\|^2$$
    $$\leq   \|\partial_s h_1 (0)\|^2  + \|\partial_s h_2 (0)\|^2 = \|\partial_s g(0)\|^2.$$
\item  From previous notations we have that $h_1 = g_1\mid_{\mathbb D_I} + g_2\mid_{\mathbb D_I}  I$ and  $h_2 = g_3\mid_{\mathbb D_I} + g_4\mid_{\mathbb D_I} I $ and as   $f \prec_{I} g$ implies,
 in the theory of functions of one complex variable, $v_1 \prec  h_1$  and $v_2  \prec  h_2$.  Subordination principle, see \cite{duren}, implies   each $z\in \mathbb D_I$  that there exist $z_1, z_2 \in \mathbb D_I$ such that
$f(z) =g_1(z_1)+g_2 (z_1)I+g_3(z_2)J+g_4(z_2)IJ $. The main result is obtained using the representation theorem for slice regular functions.
\end{enumerate}
\end{proof}
\begin{defn} Given $f\in \mathcal{R}(B(0;1))$  and let  $I\in\mathbb S$,  $0\leq r<1$ and  $1\leq  p<\infty$. Set
\begin{enumerate}
\item $ M_{\infty}(r, f) := \max\{ \|f(q)\| \ \mid \  q\in \mathbb H, \ \|q\|=r\}.$
\item $\displaystyle   M_p(r,f) :=\left[  \frac{1}{4\pi}  \int_{\partial B(0;r)}  \| f (q ) \|^p d\mu\right]^{\frac{1}{p}} ,$
where $\partial B(0;r)=\{  q\in \mathbb H \ \mid \ \|q\|=r\}$.
\item $ M_{\infty, I}(r, f) := \max\{ \|f(q)\| \ \mid \   q\in \mathbb D_I, \  \|q\|=r\}.$
\item $\displaystyle   M_{p,I}(r,f) :=\left[  \frac{1}{2\pi}  \int_{0}^{2\pi}  \| f (re^{I\theta} ) \|^p d\theta \right]^{\frac{1}{p}} .$
\end{enumerate}
\end{defn}
\begin{thm} Let $f,g \in\mathcal{R}(B(0;1))$ with $f(0)=g(0)$, such that there exists $I\in\mathbb S$ with  $f\prec_{I} g$. Then
\begin{enumerate}
\item $ M_{\infty, I}(r, f)  \leq  \sqrt{ 2 }   M_{\infty, I}(r, g) $.
\item $  \displaystyle M_{p,I}(r,f) \leq 2^{p+1}    M_{p,I}(r,g) $.
\item $M_{\infty}(r, f) \leq \sqrt{ 2 } M_{\infty, I}(r, g)\leq \sqrt{ 2 }  M_{\infty}(r, g)  $.
\item $  M_p(r,f) \leq   2^{2p+2 }\pi^2   M_{p, I}(r,g)$.
\end{enumerate}
\end{thm}
\begin{proof} All the properties are consequences of Littlewood's Subordination Theorem and  the domination of the maximum modulus, see \cite{duren}, and also of some inequalities that we prove below:
\\
 a)   $ \displaystyle \|f(z)\|^2 = \|f_1 (z)\|^2 + \|f_2(z) \|^2 \leq M_{\infty, I}(r, f_1)^2 + M_{\infty, I}(r, f_2)^2 $
 $$\leq  M_{\infty, I}(r, g_1)^2 + M_{\infty, I}(r, g_2)^2 \leq  2 M_{\infty, I}(r, g)^2,   \quad \forall  z\in \mathbb D_I,$$
where $f\mid_{\mathbb D_I} = f_1+f_2 J$ and   $g\mid_{\mathbb D_I} = g_1+g_2 J$ with $f_1,f_2,g_1,g_2 \in Hol(\mathbb D_I)$ and as $f\prec_{I} g$ then $f_k\prec_{I} g_k$ for $k=1,2$. \\
b)   {}  $   \| f (z ) \|^p \leq 2^p \left( \|f_1 (z)\|^p+ \|f_2(z) \|^p \right), \quad \forall z\in \mathbb D_I . $\\
c) In the following inequalities the   coordinates  are changed to spherical coordinates $q= (r \cos\theta_1, r\sin\theta_1\cos\theta_2, r\sin\theta_1\sin\theta_2\cos\theta_3, r\sin\theta_1\sin\theta_2\theta_3) $ where $\theta_1,\theta_2 \in [0,2\pi)$, \  $\theta_3\in [0, \pi]$, and $z=x+yI $ for $q=x+I_q y$, where $x,y\in\mathbb R$.
 $$    \displaystyle  M_p(r,f) ^p = \frac{1}{4\pi}  \int_{\mathbb S(0,r)}  \| f (q ) \|^p d\mu \leq  \frac{2^p}{4\pi}  \int_{\mathbb S(0,r)}  \| f (z ) \|^p + \| f (\bar z ) \|^p d\mu   $$
$$ \leq  \frac{2^{p+1 }\pi^2}{4\pi}  \int_{0}^{2\pi}  ( \| f (z ) \|^p + \| f (\bar z ) \|^p )  d \theta_1   = \frac{2^{p+2 }\pi^2}{4\pi}  \int_{0}^{2\pi}  \| f (z ) \|^p   d \theta_1    $$
$$   \leq   \frac{2^{2p+2 }\pi^2}{4\pi}  \int_{0}^{2\pi}  (\| f _1(z ) \|^p  +\| f _2(z ) \|^p  )  d \theta_1  \leq  \frac{2^{2p+2 }\pi^2}{4\pi}  \int_{0}^{2\pi}  (\| g_1(z ) \|^p  +\| g _2(z ) \|^p  )  d \theta_1  $$
$$ \leq \frac{2^{2p+3 }\pi^2}{4\pi}  \int_{0}^{2\pi}  (\| f (z ) \|^p  + )  d \theta_1  =   2^{2p+2 }\pi^2 M_{p, I}(r,g)$$
\end{proof}
\begin{rem}
{\rm
We see that if $f,g \in\mathcal{R}(B(0;1))$ with $f(0)=g(0)$, such that $g\prec_{\mathcal N} f$, then  all inequalities of previous theorem
are true for any $I\in \mathbb S$.}
\end{rem}

\begin{thm}(Rogosinski inequality) Let $f(q)=\sum_{k=1}^{\infty}q^{k}a_{k}$, $g(q)=\sum_{k=1}^{\infty}q^{k}b_{k}$ with $f,g\in\mathcal {R}(B(0;1)$. If $f\prec_{\mathcal N} g$, then we have
$$\sum_{k=1}^{n}\|a_{k}\|^{2}\le \sum_{k=1}^{n}\|b_{k}\|^{2}, \mbox{ for all } n=1, 2, ...,.$$
\end{thm}
\begin{proof} It is given in two cases:
\begin{enumerate}
\item If  $f, g\in \mathcal{V}(B(0;1))$ (for $\mathcal{V}$ see the notation in Remark \ref{inter}).  From $f, g\in \mathcal{V}(B(0;1))$, there exists $I, J\in \mathcal{S}$ such that $a_{k}\in \mathbb{C}_{I}$, and $b_{k}\in \mathbb{C}_{J}$, for all $k=1, 2, \ldots $.
From Definition \ref{defsub}, we have $f(q)=g(w(q))$, with $w\in {\mathcal{N}}(B(0;1))$, which immediately implies that $I=J$. Therefore,  it follows that $f, g\in \mathcal V_{I}( B(0;1))$. Applying then Theorem 6.2, p. 192 in \cite{duren} (on the slice
$\mathbb{C}_{I}$), we immediately get the required inequality.
\item Consider $f,g\in\mathcal{R}(B(0;1) )$. Let  $I, J\in \mathcal{S}$ and let $a_{1,k},a_{2,k},b_{1,k},b_{2,k}\in \mathbb{C}_{I}$ be such that $a_k=a_{1,k}+a_{2,k}J$, $b_k=b_{1,k}+b_{2,k}J$ for all $k=1, 2, ..., $. Thus $f(q)=f_1(q)+f_2(q)J $ and $g(q)=g_1(q)+g_2(q)J$ for each $q\in B(0;1)$, where    $$f_{n}(q)=\sum_{k=1}^{\infty}q^{k}a_{n,k},  \ \textrm{ and }  \  g_n(q)=\sum_{k=1}^{\infty}q^{k}b_{n,k} , \quad n=1,2.$$ As $f\prec_{\mathcal N} g$ then $f_n\prec_{\mathcal N} g_n$ for $n=1,2$. From the previous case we have
$$\sum_{k=1}^{n}\|a_{1,k}\|^{2}\le \sum_{k=1}^{n}\|b_{1,k}\|^{2}, \quad \sum_{k=1}^{n}\|a_{2,k}\|^{2}\le \sum_{k=1}^{n}\|b_{2,k}\|^{2}\mbox{ for all } n=1, 2, ...,$$
and by adding respective terms one has the result.
\end{enumerate}
\end{proof}

The following result is on independent interest, but will also be useful later.
\begin{prop} Let $g:\ B(0;R)\to \mathbb{H}$, $R>0$, be a function slice regular of the form $g(q)=\sum_{n=0}^\infty q^n a_n$.
\begin{enumerate}
\item There exists a power series
$g_r^{-\bullet}(q)=\sum_{n=0}^\infty q^n b_n$  convergent in a disc with positive radius, such that $(g\bullet g_r^{-\bullet})(q)=q$ and $g_r^{-\bullet}(0)=0$ if and only if $g(0)=0$ and $g'(0)\not=0$.
\item 
There exists a power series $g_l^{-\bullet}(q)=\sum_{n=0}^\infty q^n b_n$ convergent in a disc with positive radius,
such that $(g_l^{-\bullet}\bullet g)(q)=q$ and $g_l^{-\bullet}(0)=0$ if and only if $g(0)=0$ and $g'(0)\not=0$.
\end{enumerate}
\end{prop}

\begin{proof} (1) Assume that $g_r^{-\bullet}$ exists. Then $\sum_{n=0}^\infty\left(\sum_{m=1}^\infty q^m b_m\right) ^{*n} a_n =q$. By explicitly writing the terms of the equality
we see that we have
\begin{equation}\label{composition series}
a_0+ \left(\sum_{m=1}^\infty q^m b_m\right)a_1 + \left(\sum_{m=1}^\infty q^m b_m\right) ^{*2}a_2+\ldots +\left(\sum_{m=1}^\infty q^m b_m\right) ^{*n}a_n+\ldots =q
\end{equation}
and so to have equality it is necessary that $a_0=0$, i.e. $g(0)=0$, and $b_1a_1=1$ and so $a_1\not=0$, i. e. $g'(0)\not=0$.
To prove that the condition is sufficient, we observe that for $n\geq 2$, the coefficient of $q^n$ is zero on the right hand side of (\ref{composition series}) while on the right hand side it is
given by
\begin{equation}\label{coefficients}
b_n a_1 + P_n(b_1,\ldots , b_{n-1},a_2,\ldots ,a_n),
\end{equation}
thus we have $b_n a_1 + P_n(b_1,\ldots , b_{n-1},a_2,\ldots ,a_n)=0$, (where the polynomials $P_n$ are linear in the $a_i$'s and they contain all the possible monomials $b_{j_1}\ldots b_{j_r}$ with $j_1+\ldots +j_r=n$ and thus also $b_1^n$).   In particular we have, $b_1a_1=1$ and so $ b_1=a_1^{-1}$, $b_2a_1 +b_1^2a_2=0$ and so $b_2=-a_1^{-2}a_2a_1^{-1}$. By induction, if we have computed $b_1,\ldots ,b_{n-1}$ we can compute $b_n$ using (\ref{coefficients}) and the fact that $a_1$ is invertible and this concludes the proof. The function $g_r^{-\bullet}$ is right inverse of $g$, by its construction.\\
We now show that $g_r^{-\bullet}$ converges in a disc with positive radius following the proof of  \cite[Proposition 9.1]{cartan}. Construct a power series with real coefficients $A_n$ which is a majorant of $g$ as follows: set $\bar g(q)=q A_1-\sum_{n=2}^\infty q^n A_n$ with $A_1=\|a_1\|$ and $A_n\geq \|a_n\|$, for all $n\geq 2$. It is possible to compute the inverse  of $\bar g$ with respect to the (standard) composition to get the series $\bar g^{-1}(q)=\sum_{n=1} q^n B_n$. The coefficients $B_n$ can be computed with the formula
$$
B_n A_1 + P_n(B_1,\ldots , B_{n-1},A_2,\ldots ,A_n),
$$
analogue of (\ref{coefficients}). Then we have $B_1=A_1^{-1}=\|a_1\|^{-1}$, $B_2=A_1^{-2}(-A_2)A_1^{-1}\geq \|a_1\|^{-2}\|a_2\|\cdot \|a_1\|^{-1}=\|b_2\|$ and, inductively
$$B_n=Q_n(A_1,\ldots ,A_n)\geq Q_n(\|a_1\|,\ldots ,\|a_n\|) =\|b_n\|.$$
We conclude that the radius of convergence of $g_r^{-\bullet}$ is greater than or equal to the radius of convergence of $\bar g^{-1}$ which is positive, see \cite[p. 27]{cartan}.

(2) It can be proven with computation similar to those used to prove (1). The function $g_l^{-\bullet}$ is left inverse of $g$.
\end{proof}

 \begin{lem} \label{3.3}
\begin{enumerate}
 \item If $g, w: \, B(0;1)\to \mathbb{H}$ are slice regular functions and $g$ is univalent on $B(0;1)$ with $g(0)=0$ and $g^{\prime}(0)\not=0$, then
$g^{-\bullet}_{l}\bullet (g \bullet w)(q)=w(q)$, for all $q\in B(0;r)$ for a suitable $0<r<1$. If $w\in\mathcal N(B(0;1)$ then $r=1$.

 \item If $g:B(0;1)\to \mathbb{H}$ is a slice regular univalent function on $B(0;1)$ with $g(0)=0$, $g^{\prime}(0)\not=0$ and $w\in {\mathcal{N}}(B(0;1))$, then
$\partial_{s}(g\circ w)(0)=\partial_{s}(g)(0)\cdot \partial_{s}(w)(0)$.
\end{enumerate}
\end{lem}

\begin{proof}
(1) The assertion follows from Proposition \ref{assoc}.

(2) We have:
$$
\partial_s (g\circ w)(0)= (g\circ w)'(0) =\lim_{h\to 0}\frac{(g \circ w)(h)- (g \circ w)(0)}{h}, \quad h\in\mathbb R.
$$
Then the statement follows by standard arguments since we can write (remember that $w(0)=0$):
$$
\frac{(g \circ w)(h)- (g \circ w)(0)}{h}=\frac{g  (w(h))- g  (w(0))}{w(h)} \, \frac{w(h)}{h}.
$$
\end{proof}

\begin{prop}\label{3.4}
 If $f \prec_{\mathcal N} g$ then $f(0)=g(0)$ and $f(B(0;r))\subseteq g(B(0;1))$ for a suitable $0<r<1$.
 \end{prop}

\begin{proof}
By hypothesis, $f(q)=(g \circ w)(q))$ for all $q\in B(0;1)$, where $w\in\mathcal N(B(0;1))$. By Schwarz's lemma (Theorem 4.1 in \cite{GS}), we get $\|w(q)\|\le \|q\|$ for all $q\in B(0;1)$, which implies
$f(B(0;1))=\{(g \circ w)(q) ; q\in B(0;1)\}=\{g(\xi) ; \|\xi\|\le \|q\|, q\in B(0;1)\}\subseteq g(B(0;1))$.
\end{proof}

\begin{prop}\label{3.5}
\begin{enumerate}
 \item If $f \prec_{\mathcal N} g$ on $B(0;1)$ then $f(\overline{B(0;r)})\subset g(\overline{B(0;r)})$ for all $r\in (0, 1)$, the equality taking place if and only if $f(q)=q \lambda$, with an $\lambda\in \mathbb{H}$ satisfying $\|\lambda\|=1$ ;

 \item If $f \prec_{\mathcal N} g$ on $B(0;1)$ then $\max\{\|f(q)\| ; \|q\|\le r\}\le \max\{\|g(q)\| ; \|q\|\le r\}$ for all $r\in (0, 1)$, the equality taking place if and only if $f(q)=q \lambda$, with an $\lambda\in \mathbb{H}$ satisfying $\|\lambda\|=1$ ;

 \item If $f \prec_{{\mathcal{N}}} g$ then $\|\partial_{s}(f)(0)\|\le \|\partial_{s}g(0)\|$, the equality taking place if and only if $f(q)=q \lambda$, with an $\lambda\in \mathbb{H}$ satisfying $\|\lambda\|=1$.
\end{enumerate}
\end{prop}

\begin{proof} The proof of (1), (2) follows from the Schwarz's lemma  and reasoning  as in the proof of Proposition \ref{3.4}.

(3)  Since $w\in {\mathcal{N}}(B(0;1))$, we get $\partial_{s}(f)(0)=\partial_{s}(g)(0)\cdot \partial_{s}(w)(0)$. Since by Schwarz's lemma we have
$\|\partial_{s}(w)(0)\|\le 1$, we get the desired conclusion.
\end{proof}

\section{Starlike and Convex Slice Regular Functions}

For our considerations, firstly we need the concepts of starlike and convex sets and functions.

\begin{defn} \begin{enumerate} \item $A\subset \mathbb{H}$ is called starlike with respect to the origin $0$, if for all $t\in [0, 1]$
we have $t A\subset A$.
\item $A\subset \mathbb{H}$ is called convex if for all $t\in [0, 1]$ we have $(1-t)A + t A\subset A$.
\end{enumerate}
\end{defn}

\begin{rem}
The definitions for starlike and convex sets in $\mathbb{H}$ are in fact the standard well-known definitions in $\mathbb{R}^{4}$.
Thus, $A\subset \mathbb{R}^{4}$ with $0\in A$, is starlike if for any point $p\in A$, the (Euclidean) segment determined by $p$ and the origin
$0$ is entirely contained in $A$. Also, the convexity of $A\subset \mathbb{R}^{4}$ is understood as the property that for any two points
$p, q\in A$, the segment determined by $p$ and $q$ is entirely contained in $A$.
\end{rem}

\begin{defn}\label{4.6} Let $f\in\mathcal{R}(B(0;1))$ be a
function satisfying $f(0)=0$ and $\partial_{s}(f)(0)=1$.
\begin{enumerate}
\item  $f$ is called slice-starlike on $B(0;1)$ if for every $I\in \mathbb{S}$, we have
$${\rm Re}\left [(f(q))^{-1}\cdot q \cdot \partial_{s}(f)(q)\right ] > 0, \mbox{ for all } q=x+ Iy\in B(0;1)\bigcap \mathbb{C}_{I}.$$
\item  $f$ is called slice-convex on $B(0;1)$ if for every $I\in \mathbb{S}$, we have
$${\rm Re}\left [(\partial_{s}f(q))^{-1}\cdot q \cdot \partial_{s}^{2}(f)(q)\right ]+ 1 > 0, \mbox{ for all } q=x+ Iy\in B(0;1)\bigcap \mathbb{C}_{I}.$$
\end{enumerate}
 \end{defn}

\begin{rem}\label{4.8}
Taking into account the Splitting Lemma  in Theorem \ref{1.3}, it is clear that in Definition \ref{4.6}
the corresponding inequalities characterize in fact the starlikeness and convexity of the $\mathbb{C}_I$-valued function
$f_{I}:B(0;1)\bigcap \mathbb{C}_{I}\to \mathbb{C}_{I}$ of complex variable $z=x+ I y$, where $f_{I}(z)=f(z)=f(x +I y)$, and $x^{2}+y^{2} < 1$.
\end{rem}

For more general classes of functions we introduce the following definition.

\begin{defn}\label{4.7} Let $f\in\mathcal{R}(B(0;1))$ be a
function satisfying $f(0)=0$ and $\partial_{s}(f)(0)=1$.
\begin{enumerate}
 \item   $f$ is called starlike if $f(B(0;1)) $  is a  starlike  set with respect to the origin.
\item     $f$ is called  convex  if $f(B(0;1) $   is a convex set.
\item
 Let $I\in \mathbb S $, then   $f$ is called $I$-starlike if $f(\mathbb D_I)$ is  a starlike  subset of $\mathbb H$.
\item
Let $I\in \mathbb S $, then   $f$ is called $I$-convex  if $f(\mathbb D_I)$ is a convex set.

\end{enumerate}
\end{defn}
 Let ${\rm St}_I(B(0;1))$ be
the set of $I$-starlike slice regular functions on $B(0;1)$.
\begin{prop}
Let $J\in\mathbb S$
be any fixed element. Then
$$ \bigcap_{I\in \mathbb S} St_{I} = \{f  \in St_J \ \mid \   \ \  tf(z) = f(w)  \ \textrm{ if and only  if }  \     tf(\bar z) = f(\bar w),  z,w\in \mathbb D_J \}=:A_J. $$
\end{prop}
\begin{proof}
 The fact  that $ \bigcap_{I\in \mathbb S} St_{I} \subseteq  A_J $ follows from the definition of $A_J$. We prove that
 $A_J \subseteq  \bigcap_{I\in \mathbb S} St_{I}$.
If $  f\in  A_J$
 then
for any $q = x+I_q y\in B(0;1)$ the Representation Formula yields
$$ t f(q)  =  \frac{1}{2}\left[   (1+I_q I) tf(\bar z) +  (1- I_q I) tf(z)     \right]  $$
$$  = \frac{1}{2}\left[   (1+I_q I) f(\bar w) +  (1- I_q I) f(w)     \right] = f(r) ,$$
with $r=a+I_q b$.
\end{proof}
\begin{rem}
Let $f\in St_I  $ and  $J\in \mathbb S$ with $I\perp J$; let   $f_1,f_2 \in Hol(\mathbb D_I)$
be such that $f\mid_{D_I} =f_1+f_2J$. For any
$  z\in \mathbb D_I  $  and any $t\in[0,1]$, there exists  $  w\in \mathbb D_I  $ with
$$ tf(z) = f(w)$$
which implies
$$ tf_k(z) = f_k(w), $$
and so $f_1, f_2$ are starlike.
\end{rem}
\begin{thm}\label{5.8} Let  $f\in\mathcal{N}(B(0;1))$ be a function satisfying $f(0)=0$, $\partial_{s}(f)(0)=1$ and let $I\in\mathbb S$.
\begin{enumerate}
\item  Then, $f$ is $I$-convex  if and only if  $f$ is $J$-convex for any $J\in\mathbb S$.
\item Then, $f$ is $I$-starlike  if and only if $f$ is  starlike.
\end{enumerate}
\end{thm}
\begin{proof}Suppose that $f$ is $I$-convex. Then given  any $q,v\in D_J$
with $q=x+J y$,    $v= a+ J b$ and $x,y,a,b\in \mathbb R$, set $z =x+Iy$ and  $w=a+Ib$. Then from the Representation Formula in Theorem \ref{1.3} we have that
 $ t f(q) + (1-t) f(v)$ is equal to
$$  \frac{1}{2 } (1+J  I) t f(\bar z) +  \frac{1}{2 } (1- J  I) t f(  z)  +    \frac{1}{2 } (1+ J I) (1-t) f(\bar w  ) +  \frac{1}{2 } (1- J  I ) (1-t) f(  w  )    $$
 $$  =  \frac{1}{2 } (1+J  I) \left ( t f(\bar z)  + (1-t) f(\bar w  )  \right ) +  \frac{1}{2 } (1- J  I) \left ( t f(  z)  + (1-t) f(  w  )  \right )    $$
 $$  =  \frac{1}{2 } (1+J  I) \left (\overline {  t f( z)  + (1-t) f( w  )  } \right ) +  \frac{1}{2 } (1- J  I) \left ( t f(  z)  + (1-t) f(  w  )  \right )  $$
$$= \frac{1}{2 } (1+J  I)  \overline {f(u)}  +  \frac{1}{2 } (1- J  I) f(u)  = \frac{1}{2 } (1+J  I)  f(\bar u)  +  \frac{1}{2 } (1- J  I) f(u) =f(s) $$
where  $s=c+Jd$ if $u=c+Id \in \mathbb D_I$. We conclude that $f$ is $J$-convex. Exchanging the role of $I$ and $J$ we obtain the assertion.  \\
To show (2), we assume that $f$ is  $I$-starlike. Let $q\in B(0;1)$, with $q=x+I_qy$, $x,y\in \mathbb R$ and $I_q\in \mathbb S$. Then
 $$ t f(q) =  \frac{1}{2 } (1+J  I) t f(\bar z) +  \frac{1}{2 } (1- J  I) t f(  z)     = \frac{1}{2 } (1+J  I)  \overline {tf( z)}  +  \frac{1}{2 } (1- J  I) t f(z) $$
$$ = \frac{1}{2 } (1+J  I)  f(\bar u)  +  \frac{1}{2 } (1- J  I) f(u) =f(s) \in B(0;1), $$
where $s \in \mathbb D_{I_q}$, thus $f$ is starlike. The converse is trivially true.
\end{proof}

\noindent
From the previous result and Remark \ref{4.8} it follows that a function $\in\mathcal N (B(0;1))$ can be simply called slice-convex instead of $I$-convex.
\\
     Let $I\in\mathbb S$, denote $$ \mathcal P_I = \{ f\in \mathcal {R}(B(0;1) )  \  \mid \  {\rm Re}(f(z)) >0, \ z\in \mathbb D_I , \ f(0)=1\}.$$
 \begin{thm} Let $f\in \mathcal N(B(0;1))$ be such that  $f(0) = 0$ and  $\partial _sf(0) = 1$.
\begin{enumerate}
\item Then $f$ is a starlike function if and only if  the function $q \partial_s f(q) \ast f(q)^{-*}, \quad  q\in B(0;1)$ belongs to $\mathcal P_I$ for each  $I\in \mathbb S$.
\item Let $I\in \mathbb S$, then $f$ is $I$-convex function if and only if  the function  $  1 + q  \partial_s^2 f(q) \ast (\partial_s f(q) )^{-*} , \quad q\in B(0;1)$  belongs to $\mathcal P_I$.
\item Let $I\in \mathbb S$, then $f$ is a $I$-convex function if and only if   the function $q \ast \partial_s f(q) , \quad q\in B(0;1)$  is a $I$-starlike function.
\item If $f$ is a starlike function, then its coefficients    satisfy  $\|a_n\|\leq  n$ for
$n \in\mathbb N$.  Strict inequality holds for all $n$ unless $f $is a rotation of the Koebe
function.
\item Let $I\in \mathbb{S}$. If $f$ is a $I$-convex function, then $\|a_n\| \leq 1$  for  $ n = 2, 3,\dots $. Strict inequality holds for all
$n$ unless $f$ is a rotation of $q (1-q)^{-1} $.
\item Let $I\in \mathbb{S}$. If $f$ is a $I$-convex function then $B(0; 1/2)\cap \mathbb{C}_{I}\subset f[B(0;1)\cap \mathbb{C}_{I}]$.
\end{enumerate}
\end{thm}
\begin{proof} Since $f\in {\mathcal{N}}(B(0;1))$, we have that $f:B(0;1)\cap \mathbb{C}_{I}\to \mathbb{C}_{I}$, is a holomorphic function on $B(0;1)\cap \mathbb C_I$  for every  $I\in \mathbb{S}$.

1. According to the above Theorem \ref{5.8}, point 2 and by Theorem 2.10 in \cite{duren} (applied on every slice $\mathbb{C}_{I}$), we get the desired conclusion.

2. It is immediate by Theorem 2.11 in \cite{duren} applied on the given slice $\mathbb{C}_{I}$.

3. It is immediate by Theorem 2.12 in \cite{duren} (Alexander's theorem) applied on the given slice $\mathbb{C}_{I}$.

4. It is immediate by the above Theorem \ref{5.8}, point 2 and by Theorem 2.14 in \cite{duren} applied on any fixed slice $\mathbb{C}_{I}$.

5. It immediate from the above Theorem \ref{5.8}, point 1, and the Corollary from the page 45 in \cite{duren} applied on the slice $\mathbb{C}_{I}$.

6. It is immediate from Theorem 2.15 in \cite{duren} applied on the slice $\mathbb{C}_{I}$.
\end{proof}

We can prove the following geometric characterizations of the slice-starlike and slice-convex functions.

\begin{thm}\label{4.9} Let $f\in\mathcal{N}(B(0;1))$  be a function satisfying $f(0)=0$ and $\partial_{s}(f)(0)=1$.
If $f$ is slice-starlike on $B(0;1)$ then $f$ is univalent in $B(0;1)$ and $f[B(0;1)]$ is a starlike set. Moreover, denoting
$B(0;r)=\{q\in \mathbb{H} ; \|q\|<r\}$, also $f[B(0;r)]$ is a starlike set for every $0<r<1$.
\end{thm}

\begin{proof} Firstly, we will prove that $f$ is univalent in $B(0;1)$. In this sense, we have four cases.

Case 1). For a fixed $I\in \mathbb{S}$, let $q_{1}, q_{2}\in B(0;1)\bigcap \mathbb{C}_{I}$ be with $q_{1}\not = q_{2}$.
By the Remark \ref{4.8} and by the starlikeness of $f$, it follows (see e.g. \cite{Moc}, p. 45-46 or \cite{Rob}) the univalence of $f$ on $B(0;1)\bigcap \mathbb{C}_{I}$, that is $f(q_{1})\not=f(q_{2})$.

Case 2). Let $q_{1}=a_{1}\in \mathbb{R}, q_{2}=a_{2}\in \mathbb{R}$ be with $q_{1}\not=q_{2}$. Choosing an arbitrary $I\in \mathbb{S}$ we can write
$q_{1}=a_{1}+ I \cdot 0$, $q_{2}=a_{2}+ I \cdot 0$ and we are in Case 1), obtaining thus $f(q_{1})\not=f(q_{2})$.

Case 3). Let $q_{1}\in \mathbb{C}_{I}$, $q_{2}\in \mathbb{C}_{-I}$ be with $q_{1}\not =q_{2}$. But it is easy to see that $\mathbb{C}_{I}=\mathbb{C}_{-I}$, which means that we are in the Case 1) and we obtain $f(q_{1})\not=f(q_{2})$.

Case 4). Let $q_{1}\in B(0;1)\bigcap \mathbb{C}_{J}$ and $q_{2}\in B(0;1)\bigcap \mathbb{C}_{I}$ be with $q_{1}\not=q_{2}$, where
$I, J\in \mathbb{S}$, $J\not=I$ and $J\not=-I$. Since by Theorem \ref{1.3}, (ii) we have $f(q_{1})=a + J b$ and
$f(q_{2})=c+ I d$, if we would have $f(q_{1})=f(q_{2})$, then we would necessarily get two subcases : (i) $a=c$ and $b=d=0$ or
(ii) $a=c$ and $I=J$, $b=c\not=0$.

Subcase $(4_{i})$. It follows that $q_{1}, q_{2}\in \mathbb{R}$, which according to the Case 2) would imply $f(q_{1})\not=f(q_{2})$,
a contradiction.

Subcase $(4_{ii})$. This is impossible because it implies $I=J$, a contradiction.

Collecting the above results, it follows the univalence of $f$ in $B(0;1)$.

Now we prove that $f(B(0;1))$ is starlike. Writing $\mathbb{D}_{I}=B(0;1)\bigcap \mathbb{C}_{I}$,
we have $B(0;1)=\bigcup_{I\in \mathbb{S}}\mathbb{D}_{I}$,
$f(B(0;1))=f(\bigcup_{I\in \mathbb{S}}\mathbb{D}_{I})=\bigcup_{I\in \mathbb{S}}f(\mathbb{D}_{I})$,
which implies for all $t\in [0, 1]$,
$$t f(B(0;1)) =\bigcup_{I\in \mathbb{S}}t f(\mathbb{D}_{I})\subset \bigcup_{I\in \mathbb{S}} f(\mathbb{D}_{I})=f(B(0;1)).$$
We used here the fact that the starlikeness of $f:\, \mathbb{D}_{I}\to \mathbb{C}_{I}$ as function of $z=x+ I y\in \mathbb{D}_{I}$
implies that the set $f(\mathbb{D}_{I})$ is starlike (see e.g. \cite{Moc}, p. 44-45).

Similarly, denoting $\mathbb{D}_{I}(0; r)=B(0; r)\bigcap \mathbb{C}_{I}$ and taking into account that the starlikeness of
$f:\, \mathbb{D}_{I}(0;1)\to \mathbb{C}_{I}$ as function of $z = x + I y\in \mathbb{D}_{I}$ also implies (see again e.g. \cite{Moc}, p. 44-45)
that $f(\mathbb{D}_{I})$ is starlike, reasoning as above we immediately obtain that $f(B(0; r))$ is starlike, for every $r\in (0, 1)$.
\end{proof}

\begin{rem}
The geometric properties in Theorem \ref{4.9} are similar to those obtained in the case of a complex variable. Also, note that from the proof of
Theorem \ref{4.9} it easily follows that if $f\in\mathcal{N}(B(0;1))$ is a function satisfying $f(0)=0$, $\partial_{s}(f)(0)=1$, such that
$f(z)$ with $z\in \mathbb{C}$ is univalent in the open unit disk in $\mathbb{C}$, then as function of $q$, $f$ is univalent in $B(0;1)$.
\end{rem}

\begin{thm}\label{4.15}
Let $f\in\mathcal{N}(B(0;1))$ be a function satisfying $f(0)=0$ and $\partial_{s}(f)(0)=1$.
If $f$ is slice-convex on $B(0;1)$ then $f$ is univalent in $B(0;1)$ and $f[B(0;1)]$ is a convex set. Moreover, denoting
$B(0;r)=\{q\in \mathbb{H} ; \|q\|<r\}$, also $f[B(0;r)]$ is a convex set for every $0<r<1$.
\end{thm}

\begin{proof} The univalence of $f$ in $B(0;1)$ follows exactly as in the proof of Theorem \ref{4.9}, by taking into account that the convexity of a function of complex variable implies its univalence (see e.g. \cite{Moc}, p. 51 or \cite{Rob}).

We now prove the convexity of $f[B(0;1)]$ and of $f[B(0; r)]$ for every $0<r<1$.
First recall that  quaternionic intrinsic functions take any slice to itself, i.e. $f (\mathbb{D}_{I})\subset \mathbb{C}_I$.  Then
we use the relation obtained in the proof of Theorem \ref{4.9},
$f[B(0;1)]=\bigcup_{I\in \mathbb{S}}f(\mathbb{D}_{I})$, where every $f(\mathbb{D}_{I})$ is convex from the hypothesis and from
\cite{Moc}, p. 49-50 (that is $t f(\mathbb{D}_{I})+(1-t)f(\mathbb{D}_{I}) \subset f(\mathbb{D}_{I})$, for all $t\in [0, 1]$).
Moreover for any  quaternion $x+Iy\in B(0,1)$, we have that  if $f(x+Iy)=u+Iv\in f[B(0,1)]$ then also $u+\tilde Iv\in f[B(0,1)]$ for any other $\tilde I\in\mathbb S$, in fact $u+\tilde I v=f(x+\tilde I v)$ and $x+\tilde I y$ obviously belongs to $B(0,1)$.\\
We have to show that for any two points $f(q_1), f(q_2) \in f[B(0;1)]$ we have $tf(q_1)+(1-t)f(q_2) \in f[B(0;1)]$ for any $t\in [0,1]$.
We take $q_1=x_1+Iy_1$, $q_2=x_2+L y_2$ and $f(q_1)=u_1+Iv_1$, $f(q_2)=u_2+ L v_2$. Take $J\in\mathbb S$ such that $I, J, IJ=K$ is a basis of $\mathbb H$ and rewrite $L=\alpha_1 I +\alpha_2 J+\alpha_3 K$, with $\sum_{i=1}^3\alpha_i^2=1$ and
\[
\begin{split}
tf(q_1)+(1-t)f(q_2)&=t(u_1+Iv_1)+(1-t) (u_2+(\alpha_1 I +\alpha_2 J+\alpha_3 K)v_2)\\
&=(tu_1+(1-t)u_2)+\tilde I v(t,v_1,v_2,\alpha_1)
\end{split}
\]
where we set
$$
w=(tv_1 +(1-t)\alpha_1v_2) I +(1-t)(\alpha_2v_2 J+ \alpha_3 v_2 K), \qquad
\tilde I:=\frac{w}{\|w\|}
$$
and
\[
\begin{split}
v(t,v_1,v_2,\alpha_1)&=\|(tv_1 +(1-t)\alpha_1v_2) I +(1-t)(\alpha_2v_2 J+ \alpha_3 v_2 K)\|\\
&=(t^2v_1^2+(1-t)^2v_2^2+2t(1-t)\alpha_1 v_1v_2)^{1/2}.
\end{split}
\]
To show that $(tu_1+(1-t)u_2)+\tilde I v(t,v_1,v_2,\alpha_1)$ belongs to $f(B(0;1))$  we first observe that $|v(t,v_1,v_2,\alpha_1)|\leq | tv_1+(1-t)v_2|$.
The segment joining any point of the form $tf(q'_1)+(1-t)f(q'_2)$, for any $t\in[0,1]$ fixed,  with its conjugate $\overline{t f(q'_1)+(1-t)f(q'_2)}=tf(\bar q'_1)+(1-t)f(\bar q'_2)$, where $q'_i=x_i+\tilde Iy_i$, $i=1,2$ belongs to $f(\mathbb D_{\tilde I})$, and so to  $f(B(0;1))$, by its convexity and so also the point $(tu_1+(1-t)u_2)+\tilde I v(t,v_1,v_2,\alpha_1)$ belongs to $f(\mathbb D_{\tilde I})$ since it belongs to that segment.

Indeed, denoting $t f(q'_1)+(1-t)f(q'_2)=A + \tilde{I} B$, we have $$\overline{t f(q'_1)+(1-t)f(q'_2)}=A - \tilde{I} B$$ and
the segment $S=[t f(q'_1)+(1-t)f(q'_2), \overline{t f(q'_1)+(1-t)f(q'_2)}]$ can be written as the set
\begin{equation}\label{(1.1)}
\begin{split}
S&=\{\lambda (A+\tilde{I} B)+(1-\lambda)(A-\tilde{I} B) ; \lambda \in [0, 1]\}\\
&=\{A+\tilde{I}[\lambda B +(1-\lambda)(-B)] ; \lambda \in [0, 1]\}.
\end{split}
\end{equation}
On the other hand, denoting $f(q^{\prime}_{1})=u_{1}+\tilde{I} v_{1}$ and $f(q^{\prime}_{2})=u_{2}+\tilde{I} v_{2}$ (since $f$ is intrinsic function), for fixed $t\in [0, 1]$ we get
$$
A+\tilde{I}  B=t u_{1}+(1-t)u_{2}+\tilde{I}(t v_{1}+ (1-t) v_{2}),
 $$
 $$
 A-\tilde{I} B = t u_{1}+(1-t)u_{2}-\tilde{I}(t v_{1}+ (1-t) v_{2})$$
and by \eqref{(1.1)} the segment $S$ can be written as
$$S=\{t u_{1}+ (1-t)u_{2} + \tilde{I} [\lambda(t v_{1}+(1-t)v_{2})+(1-\lambda)(-tv_{1}-(1-t)v_{2})] ; \, \lambda \in [0, 1]\}.$$
Now, since $|v(t,v_1,v_2,\alpha_1)|\leq | tv_1+(1-t)v_2|$, it follows that there exists a $\lambda\in [0, 1]$ such that
$$v(t,v_1,v_2,\alpha_1)=\lambda(t v_{1}+(1-t)v_{2})+(1-\lambda)(-tv_{1}-(1-t)v_{2}),$$
which implies that $(tu_1+(1-t)u_2)+\tilde I v(t,v_1,v_2,\alpha_1)$ belongs to the segment $S$.
\\
The convexity of $f[B(0,r)]$, $0<r<1$ can be proven similarly.
\end{proof}

\begin{rem}
The geometric properties in Theorem \ref{4.15} are similar to those obtained in the case of a complex variable.
\end{rem}

\begin{rem}
As consequences/applications of Theorem \ref{4.9} and \ref{4.15},
we can easily
generate many examples of transformations from $\mathbb{R}^{4}$ to $\mathbb{R}^{4}$, with nice
geometric properties, because any $f:B(0,1)\to \mathbb{H}$ can be written in the form \eqref{geo}
and therefore $f$ can be also be viewed as
the transformation
$$(x_1, x_2, x_3, x_4) \to$$
$$ (P_{1}(x_1, x_2, x_3, x_4), P_{2}(x_1, x_2, x_3, x_4), P_{3}(x_1, x_2, x_3, x_4),
P_{4}(x_1, x_2, x_3, x_4)),$$
with all $P_{k}(x_1, x_2, x_3, x_4)$, $k=1, 2, 3, 4$, real-valued.
Thus, if in the expression of an analytic function  of complex variable
with the coefficients in the series expansions all real, which is a
starlike (or convex) function, we replace $z$-complex, by $q$-quaternion, we
get in fact a (infinite differentiable) mapping from $\mathbb{R}^4$ to $\mathbb{R}^4$, which
transforms the Euclidean open unit ball and all its open sub-balls (that is with the same center at origin) into
starlike (respectively convex) bodies in $\mathbb{R}^4$.
\end{rem}
Just to get a flavor, we present here a few concrete examples.
\begin{ex}
\begin{enumerate}
\item Polynomials with real coefficients, like :

a) $f(q)=q[1-q/(n+1)]^{n-1}$, $q$-quaternion, $n\in \mathbb{N}$ arbitrary, is univalent and transforms
the Euclidean open unit ball in $\mathbb{R}^4$ and all its open sub-balls, into
starlike bodies in $\mathbb{R}^4$ (because $f(z), z\in \mathbb{C}$ is starlike, see e.g. \cite{Glu-Har}) ;

b) $f(q)=q+a_{2}q^{2}/(2^{2})+...+a_{n}q^{n}/(n^{2})$, $n\in \mathbb{N}$ arbitrary,
$q$-quaternion, where all $a_{k}$ are real and satisfy the inequality
$1\ge \sum_{k=2}^{n} |a_{k}|$, is univalent and transforms the Euclidean open unit ball in $\mathbb{R}^4$ and all its open
sub-balls, into convex bodies in $\mathbb{R}^4$ (because $f(z), z\in \mathbb{C}$ is convex, see e.g. \cite{Alex});

\item If we define the elementary functions of quaternion variable,
$\exp$, $\sin$, $\cos$ and $\log$ simply replacing in the
expressions of their series expansions in $\mathbb{C}$, $z\in \mathbb{C}$ by $q$-quaternion, then we get non-polynomial
transformations with nice geometric properties, like, for example, the following :

a) $f(q)=(1-\lambda)q +\lambda \sin(q)$, $q$-quaternion, is univalent and
transform the Euclidean ball in $\mathbb{R}^4$ and all its open sub-balls, into
convex bodies in $\mathbb{R}^4$, if $\lambda\in \mathbb{R}$ and satisfies the inequality
$|\lambda|\le  4e/[3(e^{2}-1)]$ (because $f(z), z\in \mathbb{C}$ is convex, see e.g. \cite{Moc2}) ;

b) $f(q)=q+\lambda[\exp(q)-1-q-q^{2}/2]$, is univalent and transforms the Euclidean ball in $\mathbb{R}^4$
and all its open sub-balls, into convex bodies in $\mathbb{R}^4$, if $\lambda\in \mathbb{R}$
and satisfies the inequality $|\lambda|\le  2/[3(e-1)]$ (because $f(z), z\in \mathbb{C}$ is starlike, see e.g. \cite{ Moc2}),
and into starlike bodies in $\mathbb{R}^4$, if $\lambda\in \mathbb{R}$ and satisfies the
inequality $|\lambda|\le  3/[(e-2)\sqrt{10}]$ (because $f(z), z\in \mathbb{C}$ is convex in this case, see e.g. \cite{Moc2}).

\item The Koebe function $K(q)$, $q\in \mathbb{H}$, is univalent and
transforms the Euclidean open unit ball in $\mathbb{R}^4$ and all its open sub-balls, into
starlike bodies in $\mathbb{R}^4$. This is because $f(z)=z\cdot \frac{1}{(1-z)^{2}}=z+2z^{2}+3z^{3}+...$, with $|z|<1$, is an analytic starlike function.
\end{enumerate}
\end{ex}

\begin{rem}\label{4.10}
It is worth noting that the above examples of univalent, starlike and convex functions of quaternion variable, are strongly contrasting with
what happens in the case of geometric function theory for functions of several complex variables, when it is surprisingly difficult
to construct starlike or convex mappings for analytic functions of several complex variables defined in the Euclidean open unit
ball, by starting from starlike and convex functions of one complex variable (see e.g. \cite{Gra-Koh}, Subsection 6.3.3 and
Problems 6.3.2, 6.3.3,(i)).
\end{rem}

As in the case of complex variable, we have the following result.

\begin{prop}\label{4.11}
A normalized function
$f\in \mathcal{R}(B(0;1))$ is slice-convex if and only if the function $h(q)=q\cdot \partial_{s}(f)(q)$ is slice-starlike.
\end{prop}

\begin{proof}
This easily follows from the fact that in Definition \ref{4.6}, the inequality in the point 2 for $f$ is equivalent to the inequality in
the point 1 for $h(q)=q\cdot \partial_{s}(f)(q)$. Indeed, we can write
$$(h(q))^{-1}\cdot q\cdot \partial_{s}(h)(q)=(h(q))^{-1}\cdot q\cdot [\partial_{s}(f)(q)+q\cdot \partial_{s}^{2}(f)(q)]$$
$$=(\partial_{s}(f)(q))^{-1}\cdot q^{-1}\cdot q \cdot [\partial_{s}(f)(q)+q\cdot \partial_{s}^{2}(f)(q)]=1+(\partial_{s}(f)(q))^{-1}\cdot q \cdot \partial_{s}^{2}(f)(q),$$
which proves the assertion.
\end{proof}

\begin{rem} Let $f(q)=q+\sum_{k=2}^{\infty}q^{k} a_{k}$ (with $a_{k} \in \mathbb{H}$, $k=0, 1, ..., $) be slice regular and satisfying $f(0)=0$ and $\partial_{s}(f)(0)=1$. Consider the Alexander-kind operator $A(f)(q)=\int_{0}^{q} t^{-1}\cdot f(t) d t$ and the Libera operator
$L(f)(q)=2q^{-1}\int_{0}^{q}f(t)dt$, defined for $q\in B(0 ; 1)$.

Firstly we get $A(f)(q)=q+\sum_{k=2}^{\infty}\frac{1}{k^{2}}q^{k} a_{k}$. Since $q \cdot \partial_{s}[A(f)(q)]=f(q)$ for all $q\in B(0;1)$, by
Lemma \ref{4.11} it follows that $f$ is slice-starlike if and only if the Alexander integral operator $A(f)$ is slice-convex on $B(0;1)$. In other words,
the Alexander integral operator transforms a slice starlike function into a slice convex function on $B(0;1)$

Secondly, we get $L(f)(q)=2 q^{-1}\int_{0}^{q}f(t) dt = q + 2\sum_{k=2}^{\infty}\frac{q^{k}}{k+1}\cdot a_{k}$.
Now, suppose that $f$ is slice-starlike in $B(0;1)$ and that $f\in\mathcal{N}(B(0;1))$. It follows that the coefficients $a_{k}\in \mathbb{R}$,
for all $k\ge 2$ and then since $f(z)$ in starlike on the open unit disk of the complex plane, by the result in \cite{Lib} we get that $L(f)(z)$ is starlike in the open unit disk of the complex
plane. Taking into account Theorem \ref{4.9}, it follows that $L(f)(q)$ is slice starlike in the ball $B(0;1)$.
\end{rem}

\section{Spirallike Slice Regular Functions}

In what follows, we deal with spirallike functions of quaternion variable. In this sense, firstly we introduce the following
well-known concept.
\begin{defn} (see e.g. \cite{Gra-Koh}, Remark 6.4.11) If $X$ is a linear space over $\mathbb{C}$, then $A\subset X$ is called spirallike
of type $\gamma\in \left (-\frac{\pi}{2}, \frac{\pi}{2}\right )$, if for all $w\in A$ and all $t\ge 0$ we have
$e^{-t \lambda}w \in A$, where $\lambda =e^{-i \gamma}$.
\end{defn}
\begin{rem}
The curve in $X$ of equation $s(t)=e^{-t \lambda}w_{0}$, with $\lambda=e^{-i \gamma}$ and $w_{0}\in X$ fixed, is called
$\gamma$-spiral in $X$ that joins $w_{0}$ with the origin. Therefore, $A\subset X$ is $\gamma$-spirallike if for any point $w_{0}\in A$,
the $\gamma$-spiral that joins $w_{0}$ and the origin, entirely belongs to $A$. When $X=\mathbb{C}$ then $s(t)$, $t\ge 0$
becomes the well-known logarithmic spiral in the plane that joins $w_{0}$ with the origin.
\end{rem}
\begin{rem}
Since $\mathbb{H}$ obviously is a linear space over $\mathbb{C}$, the parametric equations of the $\gamma$-spiral in $\mathbb{H}$,
$s(t)=e^{-t \lambda}q_{0}$, where $\lambda=e^{-i \gamma}$ and $q_{0}=x_{0}+iy_{0}+jz_{0}+ku_{0}\in \mathbb{H}$, can easily derived by simple calculation as
$$x(t)=A(t)\{\cos[B(t)]x_{0}-\sin[B(t)]y_{0}\},\, y(t)=A(t)\{\cos[B(t)]y_{0}+\sin[B(t)]x_{0}\},$$
$$x(t)=A(t)\{\cos[B(t)]z_{0}-\sin[B(t)]u_{0}\},\, y(t)=A(t)\{\cos[B(t)]u_{0}+\sin[B(t)]z_{0}\},$$
$$t\ge 0,$$
where $A(t)=e^{-t \cos(\gamma)}$ and $B(t)=t\sin(\gamma)$.
\end{rem}
Now, we are in position to introduce the concept of spirallike function.

\begin{defn}\label{5.1}
Let $f\in\mathcal{R}(B(0;1))$ be a function satisfying $f(0)=0$ and $\partial_{s}(f)(0)=1$.
Also, let $\gamma\in \left (\frac{-\pi}{2}, \frac{\pi}{2}\right )$.

Then, $f$ is called slice-spirallike of $\gamma$-type on $B(0;1)$, if for every $I\in \mathbb{S}$, we have
$$Re\left [e^{-I \gamma}\cdot q \cdot \partial_{s}(f)(q)\cdot \frac{1}{f(q)}\right ] > 0, \mbox{ for all } q=x+ Iy\in B(0;1)\bigcap \mathbb{C}_{I},$$
where $e^{-I \gamma}=\cos(\gamma)-I \sin(\gamma)$.
\end{defn}

\begin{rem}
Evidently that for $\gamma=0$, we recapture the concept of slice-starlike function in Definition \ref{4.6}, (i).
\end{rem}

We have:

\begin{thm}\label{5.2}
Let $f\in\mathcal{N}(B(0;1))$ be a function satisfying $f(0)=0$ and $\partial_{s}(f)(0)=1$ and let $\gamma\in \left (\frac{-\pi}{2}, \frac{\pi}{2}\right )$.
If $f$ is slice-spirallike of type $\gamma$ on $B(0;1)$ then $f$ is univalent in $B(0;1)$ and $f[B(0;1)]$ is a spirallike set of $\gamma$-type.
\end{thm}

\begin{proof} The univalence of $f$ in $B(0;1)$ follows exactly as in the proof of Theorem \ref{4.9}, by taking into account that the spirallikeness of $\gamma$-type of a function of complex variable implies its univalence.

It was remained to prove the spirallikeness of $\gamma$-type of $f[B(0;1)]$. With the notations in the proofs of Theorems \ref{4.9} and \ref{4.15},
we have
$B(0;1)=\bigcup_{I\in \mathbb{S}}\mathbb{D}_{I}$ and
$f[B(0;1)]=\bigcup_{I\in \mathbb{S}}f(\mathbb{D}_{I})$, where every $f(\mathbb{D}_{I})$ is $\gamma$-spirallike according to Definition \ref{5.1} and
by \cite{Moc}, p. 42, Theorem 4.4.1.

Taking $e^{-\lambda t}$ with $\lambda=e^{-i \gamma}$, this implies, for all $t\ge 0$,
$$e^{-\lambda t} f[B(0;1)] =\bigcup_{I\in \mathbb{S}}e^{-\lambda t} f(\mathbb{D}_{I}) \subset \bigcup_{I\in \mathbb{S}} f(\mathbb{D}_{I})=f(B(0;1)),$$
which proves the theorem. \end{proof}

\begin{rem}
As in the cases of starlikeness and convexity, we can easily construct spirallike functions of quaternion variable, from spirallike functions
of complex variable whose series expansions have all coefficients real, simply by replacing into their expression $z\in \mathbb{C}$, by
$q$-quaternion.
\end{rem}

\noindent{\bf Open question.}
An interesting question would be to find large subclasses of functions in $\mathcal{R}(B(0;1))$ different from the class $\mathcal{N}(B(0;1))$, defined as in Definitions \ref{4.6} and \ref{5.1}, for which the univalence and the geometric properties in Theorems \ref{4.9}, \ref{4.15}
and \ref{5.2} still hold.
As some very particular cases, for example when $f$ is of the form $f(q)=h(q)\cdot C_{0}$,\,\, $q\in B(0;1)$, with $h\in \mathcal{N}(B(0;1))$ and $C_{0}\in \mathbb{H}$ a constant, clearly $f\not \in \mathcal{N}(B(0;1))$ and if $h$ satisfies one of the Definitions \ref{4.6}, (i), (ii), or \ref{5.1}, then it is easy to see that $f$ also satisfies the same kind of definition. At the same time, since $h$ has one of the corresponding geometric property (including univalence) in Theorems \ref{4.9}, \ref{4.15}
and \ref{5.2}, it is easy to check directly that $f=h\cdot C_{0}$ keeps the same geometric property (and univalence) of $h$.

\end{document}